\documentclass [11pt,oneside,a4paper]{amsart}
\usepackage{geometry}
\newgeometry{margin=1.1in}

\usepackage{amsfonts, amsmath, amssymb, amsthm,mathtools, stmaryrd}
\usepackage{verbatim}
\usepackage[normalem]{ulem}
\usepackage{tikz-cd}
\usepackage{enumitem}

\usepackage{array}

\usepackage{hyperref}

\theoremstyle{plain}
\newtheorem{thm}{Theorem}[section]
\newtheorem{cor}[thm]{Corollary}

\newtheorem{prop}[thm]{Proposition}

\numberwithin{equation}{section}

\newtheorem{conjecture}{Conjecture}  
\newtheorem{conj}[conjecture]{Conjecture}  

\theoremstyle{definition}
\newtheorem{defn}[thm]{Definition}
\newtheorem{example}[thm]{Example}
\newtheorem{lemma}[thm]{Lemma}
\newtheorem{rmk}[thm]{Remark}

\theoremstyle{remark}

\newcommand{\BC}{{\mathbb{C}}}

\newcommand{\BE}{{\mathbb{E}}}
\newcommand{\BF}{{\mathbb{F}}}
\newcommand{\BG}{{\mathbb{G}}}
\newcommand{\BH}{{\mathbb{H}}}

\newcommand{\BQ}{{\mathbb{Q}}}
\newcommand{\BR}{{\mathbb{R}}}

\newcommand{\BT}{{\mathbb{T}}}

\newcommand{\BZ}{{\mathbb{Z}}}

\newcommand{\CC}{{\mathcal C}}

\newcommand{\CE}{{\mathcal E}}
\newcommand{\CF}{{\mathcal F}}

\newcommand{\CH}{{\mathcal H}}
\newcommand{\CI}{{\mathcal I}}

\newcommand{\CK}{{\mathcal K}}
\newcommand{\CL}{{\mathcal L}}

\newcommand{\CO}{{\mathcal O}}
\newcommand{\CP}{{\mathcal P}}
\newcommand{\CQ}{{\mathcal Q}}

\newcommand{\Fa}{{\mathfrak{a}}}

\newcommand{\Ft}{{\mathfrak{t}}}

\newcommand{\blangle}{\big\langle}
\newcommand{\brangle}{\big\rangle}

\newcommand{\pt}{{\mathsf{p}}}
\newcommand{\td}{{\mathrm{td}}}

\newcommand{\A}{\mathsf{A}}

\DeclareFontFamily{OT1}{rsfs}{}
\DeclareFontShape{OT1}{rsfs}{n}{it}{<-> rsfs10}{}
\DeclareMathAlphabet{\curly}{OT1}{rsfs}{n}{it}
\renewcommand\hom{\curly H\!om}
\newcommand\ext{\curly Ext}
\newcommand\Ext{\operatorname{Ext}}
\newcommand\Hom{\operatorname{Hom}}

\newcommand{\p}{\mathbb{P}}

\newcommand\Spec{\operatorname{Spec}}

\newcommand{\vir}{{\text{vir}}}

\newcommand{\Coh}{\mathrm{Coh}}

\newcommand{\Pic}{\mathop{\rm Pic}\nolimits}
\newcommand{\PT}{\mathsf{PT}}
\newcommand{\piPT}{\pi\textup{-}\mathsf{PT}}

\newcommand{\DT}{\mathsf{DT}}
\newcommand{\GW}{\mathsf{GW}}

\newcommand{\Mod}{\mathsf{Mod}}

\newcommand{\T}{\mathsf{T}}
\newcommand{\Hilb}{\mathsf{Hilb}}

\newcommand{\id}{\mathrm{id}}

\newcommand{\Ker}{\mathrm{Ker}}

\newcommand{\tch}{\widetilde{\mathsf{ch}}}
\newcommand{\ch}{\mathsf{ch}}

\newcommand{\wt}{\mathsf{wt}}
\newcommand{\ind}{\mathsf{ind}}

\newcommand{\QJac}{\mathsf{QJac}}

\newcommand{\Jac}{\mathsf{Jac}}
\newcommand{\SL}{\mathrm{SL}}

\newcommand{\pr}{\mathrm{pr}}
\DeclareMathOperator{\Wt}{\mathsf{WT}}

\newcommand\Tan{\mathrm{Tan}}
\newcommand\Res{\operatorname{Res}}
\newcommand{\red}{\mathsf{red}}

\begin{document}
\baselineskip=14.5pt
\title[PT theory of elliptic threefolds]{Pandharipande-Thomas theory of elliptic threefolds, quasi-Jacobi forms and holomorphic anomaly equations}

\author{Georg Oberdieck}

\address{Department of Mathematics, KTH Royal Institute of Technology}
\email{georgo@kth.se}

\author{Maximilian Schimpf}
\email{mschimpf@kth.se}

\date{\today}

\begin{abstract}
Let $\pi : X \to B$ be an elliptically fibered threefold satisfying $c_3(T_X \otimes \omega_X)=0$. We conjecture that the $\pi$-relative generating series of Pandharipande-Thomas invariants of $X$ are quasi-Jacobi forms and satisfy two holomorphic anomaly equations. For elliptic Calabi-Yau threefolds our conjectures specialize to the Huang-Katz-Klemm conjecture. The proposed formulas constitute the first case of holomorphic anomaly equations in Pandharipande-Thomas theory.

We prove our conjectures for the equivariant Pandharipande-Thomas theory of $\BC^2 \times E$ when specialized to the anti-diagonal action. For $K3 \times \BC$ we state reduced versions of our conjectures. As a corollary we find an explicit conjectural formula for the stationary theory generalizing the Katz-Klemm-Vafa formula for K3 surfaces. Further evidence is available for $\p^2 \times E$ based on earlier work of the second author.

To deal with elliptic threefolds with $c_3(T_X \otimes \omega_X) \neq 0$ we show that the moduli space of $\pi$-stable pairs is represented by a proper algebraic space. We conjecture that the associated $\pi$-stable pair invariants form quasi-Jacobi forms.
\end{abstract}

\maketitle

\setcounter{tocdepth}{1} 
\tableofcontents

\section{Introduction}
\subsection{Pandharipande-Thomas theory}
Let $X$ be a smooth complex projective threefold.
A {\em stable pair} $(F,s)$ on $X$ consists of
\begin{itemize}
\item a pure $1$-dimensional sheaf $F$,
\item a section $s \in H^0(X,F)$ with zero-dimensional cokernel.
\end{itemize}
Let $P_{n,\beta}(X)$ be the fine projective moduli space of stable pairs with numerical data
\[ \mathrm{ch}_2(F) = \beta \in H_2(X,\mathbb{Z}), \quad \chi(F) = n \in \mathbb{Z}. \]
The moduli space carries naturally the descendent cohomology classes\footnote{By our conventions, we have
$\ch_0(\gamma) = - \int_X \gamma$ and $\ch_1(\gamma) = 0$.}
\[ \ch_{k}(\gamma) = \pi_{P \ast}( \mathrm{ch}_k(\mathbb{F} - \mathcal{O}) \cup \pi_X^{\ast}(\gamma) ), \quad k \geq 0, \quad \gamma \in H^{\ast}(S,\mathbb{Q}), \]
where $\pi_{P}, \pi_X$ are the projections of $P_{n,\beta}(X) \times X$ to the factors and $(\mathbb{F},s)$ is the universal stable pair.
Following \cite{MOOP} we will also use the modified descendents
\[ \tch_k(\gamma) := \ch_k(\gamma) + \frac{1}{24} \ch_{k-2}( \gamma \cdot c_2(T_X) ). \]

The Pandharipande-Thomas (PT) invariants of $X$ are defined in \cite{PT} by integrating
these classes over the virtual fundamental class of the moduli space:
\[
\left\langle \ch_{k_1}(\gamma_1) \cdots \ch_{k_n}(\gamma_n) \right\rangle^{X, \PT}_{n,\beta}
=
 \int_{[ P_{n,\beta} (X) ]^{\text{vir}}} \prod_i \ch_{k_i}(\gamma_i).
\]
We often drop the supscripts $X$ and $\PT$ from notation.
We extend the bracket $\left\langle - \right\rangle^{\PT}_{n,\beta}$
linearly to the algebra generated by the descendent classes $\ch_k(\gamma)$.
We refer to  \cite{Pdesc} for an overview of recent results and conjectures in PT theory.

\subsection{Elliptic threefolds}
Assume from now on that the threefold $X$ admits an elliptic fibration\footnote{i.e. a flat morphism with $\omega_{\pi}$ trivial on all fibers}
to a smooth projective surface $B$,
\[ \pi : X \to B. \]
We assume moreover that:
\begin{itemize}
\item $\pi$ has a section $\iota : B \to X$ with image the divisor $B_0 \subset X$,
\item $\pi : X \to B$ is a Weierstra{\ss} model \cite{D}.
\end{itemize}
Consider for any $\beta \in H_2(B,\BZ)$
the partial generating series
\begin{multline}
\label{correlator}
\left\langle \ch_{k_1}(\gamma_1) \cdots \ch_{k_n}(\gamma_n) \right\rangle_{\beta}^{X, \PT, \pi} \\
:= 
q^{-\frac{1}{2} c_1(N_{B/X}) \cdot \beta} \sum_{\substack{\widetilde{\beta} \in H_2(X,\BZ) \\ \pi_{\ast} \widetilde{\beta} = \beta }}
\sum_{m \in \frac{1}{2} \BZ}
i^{2m} p^m
  q^{B_0 \cdot \widetilde{\beta}} \left\langle \ch_{k_1}(\gamma_1) \cdots \ch_{k_n}(\gamma_n) \right\rangle^{X, \PT}_{m+\frac{1}{2} d_{\widetilde{\beta}}, \widetilde{\beta}}
\end{multline}
where
\begin{itemize}
\item $p,q$ are formal variables and $i = \sqrt{-1}$,
\item $N_{B/X}$ is the normal bundle of the section $\iota : B \to X$,
\item $d_{\widetilde{\beta}} = \int_{\widetilde{\beta}} c_1(T_X)$.
\end{itemize}
Define also the normalized series
\begin{equation} \label{normalized correlators}
Z_{\beta}\left( \ch_{k_1}(\gamma_1) \cdots \ch_{k_n}(\gamma_n) \right)
:= 
\frac{ \left\langle \ch_{k_1}(\gamma_1) \cdots \ch_{k_n}(\gamma_n) \right\rangle^{X, \PT,\pi}_{\beta} }{\left\langle 1 \right\rangle^{X, \PT,\pi}_{0} }.
\end{equation}

\vspace{5pt}
For elliptic {\em Calabi-Yau} threefolds $X$ the moduli space $P_{n,\beta}(X)$ is of virtual dimension zero
and all correlators for class $\beta$ are determined by the single Laurent series
\[ Z_{\beta} := Z_{\beta}(1) \in \BC((p))((q)). \]
The series $Z_{\beta}$ were studied intensively in physics \cite{BCOV, KMW, AS, KMRS, HKLV}
which culminated in the following beautiful conjecture
by Huang, Katz and Klemm \cite{HKK}:

\begin{conj}[Huang-Katz-Klemm, \cite{HKK}] \label{conj:HKK}
Let $X \to B$ be an elliptic Calabi-Yau threefold. For $\beta \in H_2(X,\BZ)$ the relative partition function
$Z_{\beta}$ is a meromorphic Jacobi form of weight $0$ and index $\frac{1}{2} \beta \cdot (\beta + K_B)$, which is of the form
\[
Z_{\beta}
=
	\Delta(q)^{\frac{1}{2} K_B \cdot \beta} 
	\sum_{\alpha = (\beta_1, \ldots, \beta_{\ell})}
	\frac{\varphi_{\alpha}(p,q)}
	{ \prod_{i=1}^{\ell} \Theta(p^{\mathrm{div}(\beta_i)}, q\big)^2 }
	\]
	where
	\begin{itemize}
\item $\alpha$ runs over all decompositions
	$\beta = \beta_1 + \ldots \beta_l$ into effective classes $\beta_i \in H_2(B,\mathbb{Z})$
	which are of divisibility $\mathrm{div}(\beta_i)$ in $H_2(B,\mathbb{Z})$, and
\item $\varphi_{\alpha} \in \Jac$ are (holomorphic) weak Jacobi forms.
	\end{itemize}
\end{conj}

In the statement above we used the weight $12$ discriminant modular form
\[ \Delta(q) = q \prod_{n \geq 1} (1-q^n)^{24} \quad \in \Mod_{12} := \Jac_{12,0} \]
and the index $1/2$ weight $-1$ Jacobi theta function
\begin{equation} \label{Theta}
\Theta(p,q) 
=  (p^{1/2}-p^{-1/2})\prod_{m\geq 1} \frac{(1-pq^m)(1-p^{-1}q^m)}{(1-q^m)^{2}}.
\end{equation}
The bigraded algebra of weak Jacobi forms $\Jac$ is recalled in Section~\ref{subsec:quasi-Jacobi forms}.

\vspace{4pt}
\begin{example}
Let $\pi : X \to \p^2$ be the elliptic CY3 over $\p^2$
and identify $\Pic(\p^2) = \BZ$.
The first non-reduced case in the base is degree $2$.
Conjecture~\ref{conj:HKK} says 
\[
Z_{2}
\, = \,
\frac{\varphi_{11}(p,q)}{\Delta^{3} \varphi_{-2,1}(p,t)^2}
+
\frac{\varphi_{2}(p,q)}{\Delta^{3} \varphi_{-2,1}(p^2,t)}.
\]
for 
$\varphi_{11} \in \Jac_{32,1}$ and $\varphi_{2} \in \Jac_{34,3}$
which are easily fixed by basic computations, see \cite{HKK}. \qed
\end{example}

Conjecture~\ref{conj:HKK} allows for efficient computations of Pandharipande-Thomas invariants
in these geometries. For every given class $\beta$ one only has to fix a short list of coefficients
to determine the series $Z_{\beta}$.
This has been used in \cite{HKK} to determine the invariants of the elliptic CY3 $X \to \p^2$ up to degree 20 over the base.
Partial evidence towards this conjecture on the mathematical side can be found in \cite{FM, BK, elfib}.

In this paper, we investigate the generating series \eqref{correlator} for elliptic threefolds which are not necessarily Calabi-Yau.
The cohomological insertions $\ch_k(\gamma)$ are a new feature compared to the Calabi-Yau case.
We will conjecture the appearance of quasi-Jacobi forms and holomorphic anomaly equations.
The outcome is similar to the situation for elliptic genera:
Elliptic genera of Calabi-Yau manifolds are Jacobi forms, while elliptic genera of complex manifolds are quasi-Jacobi forms  in general \cite{Lib}.

\subsection{Main conjectures}
We state our main conjectures.

Our first conjecture concerns the normalization factor:
\begin{conj}[Normalization factor] \label{conj:normalization} We have
	\begin{align*}
		\left\langle 1 \right\rangle^{X,\PT, \pi}_{0} & =
		\prod_{m \geq 1} (1-q^m)^{-e(B) - c_1(N) \cdot (c_1(T_B) + c_1(N))} \prod_{\ell, m \geq 1} (1- p^{\ell} q^{m} )^{-\ell \cdot c_3(T_X \otimes \omega_X) }
	\end{align*}
	where $N = N_{B/X}$ is the normal bundle of the section.
\end{conj}

For elliptic Calabi-Yau threefolds Conjecture~\ref{conj:normalization} is known by work of Toda \cite[Thm 6.9]{T12}.
We verify it below for the product of an elliptic surface with a curve,
and for the trivial elliptic fibration, see Proposition~\ref{prop:PT0 check}.


Next we consider non-zero classes $\beta$.
We specialize to the case that
\[ c_3(T_X \otimes \omega_X) = 0 \]
which is satisfied if and only if $c_1(N_{B/X})^2 = 0$,
see Corollary~\ref{cor:constant}.
In particular, this applies to the product of an elliptic surface with a curve or the trivial elliptic fibration.

Let $\QJac = \oplus_{k,n} \QJac_{k,m}$ be the algebra of quasi-Jacobi forms,
bigraded by weight $k$ and index $m$ with finite-dimensional summands
$\QJac_{k,m}$, see Section~\ref{subsec:quasi-Jacobi forms}.

We have the following (partial) generalization of the Huang-Katz-Klemm conjecture:\footnote{For Calabi-Yau threefolds $X$ we have $N_{B/X} \cong \omega_X$.}

\begin{conj} \label{conj:QJac introduction}
Let $X \to \p^1$ be an elliptic threefold satisfying $c_3(T_X \otimes \omega_X)=0$. Then:
\begin{enumerate}
\item[(a)] (Quasi-Jacobi form property) Each $Z_{\beta}( \tch_{k_1}(\gamma_1) \cdots \tch_{k_n}(\gamma_n))$ is a meromorphic quasi-Jacobi form of
	\begin{align*}
		\text{weight} \quad & K_X \cdot \iota_{\ast}\beta + \sum_{i=1}^{n} (k_i - 1 + \mathrm{wt}(\gamma_i)), \\
		\text{index} \quad & \frac{1}{2} \beta \cdot (\beta + c_1(N_{B/X})),
	\end{align*}
where $\wt(\gamma_i)$ are the eigenvalues of the eigenvectors $\gamma_i \in H^{\ast}(X)$ under the weight operator
\[ \Wt := [ B_0 \cup  , \pi^{\ast} \pi_{\ast} ] : H^{\ast}(X) \to H^{\ast}(X). \]
\item[(b)] (Pole structure) We have
	\[
	Z_{\beta}\left( \tch_{k_1}(\gamma_1) \cdots \tch_{k_n}(\gamma_n) \right)
	=
	\Delta(q)^{\frac{1}{2} c_1(N) \cdot \beta} 
	\sum_{\alpha = (\beta_1, \ldots, \beta_{\ell})}
	\frac{\varphi_{\alpha}(p,q)}
	{ \prod_{i=1}^{\ell} \Theta(p^{\mathrm{div}(\beta_i)}, q\big)^2 }
	\]
	where
	\begin{itemize}
\item $\alpha$ runs over all decompositions
	$\beta = \beta_1 + \ldots \beta_l$ into effective classes $\beta_i \in H_2(B,\mathbb{Z})$
	which are of divisibility $\mathrm{div}(\beta_i)$ in $H_2(B,\mathbb{Z})$, and
\item $\varphi_{\alpha} \in \QJac$ are quasi-Jacobi forms.
	\end{itemize}
\end{enumerate}
\end{conj}


The algebra $\QJac$ admits an embedding
into the free polynomial ring of two particular generators over the algebra of (weak) Jacobi forms,
\[ \QJac \subset \Jac[ \A , G_2]. \]
The Eisenstein series $G_2(q)$ and the Weierstra{\ss} zeta function $\A(p,q) = p\frac{d}{dp} \log \Theta(p,q)$ are recalled in Section~\ref{subsec:quasi-Jacobi forms}.
In particular, by viewing a quasi-Jacobi form as a polynomial in $\A$ and $G_2$ with coefficients in $\Jac$,
we can talk about the derivation operators:
\[
\frac{d}{dG_2} : \QJac \to \QJac, \quad \frac{d}{d \A}: \QJac \to \QJac.
\]

We state the holomorphic anomaly equations which fixes the dependence of the quasi-Jacobi form
$Z_{\beta}\left( \tch_{k_1}(\gamma_1) \cdots \tch_{k_n}(\gamma_n) \right)$
on the quasi-parameters, and hence determines the series up to a purely Jacobi form part.

\begin{conj}[Holomorphic anomaly equations] \label{conj:HAE introduction} Assume $c_3(T_X \otimes \omega_X) = 0$ and the previous conjecture. Then we have:
	\begin{align*}
		\frac{d}{d \A} Z_{\beta}\left( \tch_{k_1}(\gamma_1) \cdots \tch_{k_n}(\gamma_n) \right)
		= & 
		\sum_{i = 1}^{n}
		\sigma_i 
		Z_{\beta}\left( \tch_{k_i-1}( \gamma_i \Delta_{B,1} ) \tch_{2}( \Delta_{B,2} ) \prod_{\ell \neq i} \tch_{k_{\ell}}(\gamma_{\ell}) \right) \\
		& +\sum_{i=1}^{n} \sigma_i Z_{\beta} \left( \ \tch_{k_i+1}( \pi^{\ast} \pi_{\ast}(\gamma_i) ) \prod_{\ell \neq i} \tch_{k_{\ell}}(\gamma_{\ell}) \ \right).
	\end{align*}
		
and
		\begin{small}
\begin{align*}
& \frac{d}{dG_2} Z_{\beta}\left( \tch_{k_1}(\gamma_1) \cdots \tch_{k_n}(\gamma_n) \right) = \\
& -2 \sum_{i < j} \sigma_{ij} Z_{\beta}\left( \tch_{k_i-1}(\gamma_i \Delta_{B,1}) \tch_{k_j-1}(\gamma_j \Delta_{B,2}) \prod_{\ell \neq i,j} \tch_{k_{\ell}}(\gamma_{\ell}) \right) \\
& - \sum_{i=1}^{n} \sigma_i Z_{\beta}\left( \tch_{k_i-2}( \gamma_i \cdot c_2(B)) \prod_{\ell \neq i} \tch_{k_{\ell}}(\gamma_{\ell}) \right) \\
& + \sum_{i=1}^{n} \sum_{m_1 + m_2 = k_i} \sigma_i Z_{\beta}\left( (-1)^{(1+\wt_L)(1+\wt_R)} \frac{(m_1 - 1 + \wt^L)! (m_2 - 1 + \wt^R)!}{(k_i - 2 + \wt(\gamma_i))!} \tch_{m_1} \tch_{m_2} ( \CE( K_X \cdot \gamma_i )) \prod_{\ell \neq i} \tch_{k_{\ell}}(\gamma_{\ell}) \right) \\
& -2 \sum_{i < j} \sigma_{ij} (-1)^{(1+\wt(\gamma_i)) (1+\wt(\gamma_j))} \binom{ k_i + k_j - 4 + \wt(\gamma_i) + \wt(\gamma_j) }{ k_i-2 + \wt(\gamma_i), \, k_j - 2 + \wt(\gamma_j) } 
 Z_{\beta}\left( \tch_{k_i+k_j-2}( \CE( \gamma_i \boxtimes \gamma_j )) \prod_{\ell \neq i,j} \tch_{k_{\ell}}(\gamma_{\ell}) \right),
\end{align*}
\end{small}
where 
\begin{itemize}
\item $(\Delta_{B,1}, \Delta_{B,2})$ stands for summing over the K\"unneth decomposition of the diagonal class $\Delta_B \in H^{\ast}(B \times B)$,
\item we let $\mathcal{E} \in H^{\ast}(X \times X \times X)$ be the correspondence
defined by
\begin{align*} \mathcal{E} & \coloneqq \Delta_{X,12} \cdot \pi^*\Delta_{B,13}+\Delta_{X,13} \cdot \pi^*\Delta_{B,12} + \Delta_{X,23} \cdot \pi^*\Delta_{B,12}  \\
	& \quad -\pi^*\Delta_{B,123}\big(\mathrm{pr}_1^*W + \mathrm{pr}_2^* W + \mathrm{pr}_3^* W \big), \end{align*}
\item $\sigma_i$ (resp. $\sigma_{ij}$) are the signs obtained by permuting the $i$-th entry (resp. the $i$-th and $j$-th entry) of $(\gamma_1, \ldots, \gamma_n)$ to the left-most position,
\item and several other conventions explained in Section~\ref{sec:HAE}.
\end{itemize}
\end{conj}

For elliptic Calabi-Yau threefolds, Conjectures~\ref{conj:QJac introduction} and~\ref{conj:HAE introduction} together specialize to Conjecture~\ref{conj:HKK}.
Evidence in non-Calabi-Yau geometries is given in Section~\ref{sec:First examples} for $\p^2 \times E$ in degrees $d=1,2$ based on work of the second author \cite{Schimpf}.
Then we consider the equivariant theory of $\BC^2 \times E$.
Our conjectures do not apply here directly.
However, in general we expect Conjecture~\ref{conj:HAE introduction} to hold also when the target $X$ admits a action by a torus $T$ with $P_{n,\beta}(X)^T$ compact and the PT theory is taken equivariantly with respect to the torus action (i.e. is defined by equivariant residues), 
{\em provided} that we have the vanishing of the {\em class}\footnote{This is stronger than the vanishing of
$\int_{X} c_3(T_X \otimes \omega_X) \in H_T(pt)$.}
\[ c_3(T_X \otimes \omega_X)  = 0 \in H_{T}(X). \]
For $\BC^2 \times E$ this holds for the anti-diagonal action
and in Section~\ref{ex:C2xE} we will prove Conjecture~\ref{conj:HAE introduction} in this case.
Then in Section~\ref{sec:extended example K3xC} we consider the equivariant PT theory of $S \times \BC$ where $S$ is a K3 surface.
We formulate reduced versions of the holomorphic anomaly equation,
and use it to find explicit conjectural formulas for the full stationary theory.

Conjecture~\ref{conj:HAE introduction} was found by moving the holomorphic anomaly equations
for Gromov-Witten invariants of elliptic fibrations conjectured in \cite{elfib,HAE}
via the GW/PT descendent correspondence \cite{PaPix, MOOP} to the PT side.
Concretely, according to \cite{PaPix} one can express every Gromov-Witten series in a universal way as a Laurent polynomial in $z$ with coefficients given by PT series under the variable change $p=e^z$. Explicit formulas for this have been given for toric threefolds and essential descendents in \cite{MOOP}. By expressing the factorwise $G_2$-derivative of a quasi-Jacobi form in terms of its derivatives with respect to $G_2$ and $\A$ (see \cite[Lemma.2.15]{HilbHAE}) we can derive Conjecture \ref{conj:HAE introduction} from \cite[Conj.~B]{HAE} and Conjecture \ref{conj:QJac introduction} by comparing $z$-coefficients.\footnote{This uses that $z$ and $p=e^z$ are algebraically independent.}
Altogether we obtain explicit holomorphic anomaly equations for PT brackets whenever all insertions involved are essential.
This suggested the general formula presented above, but since there are no nontrivial cases where \cite[Conj.~B]{HAE}, Conjecture \ref{conj:QJac introduction} and an explicit form of the GW/PT correspondence are all simultaneously known, this argument is almost never rigorous.

\subsection{$\pi$-stable pairs invariants}
In the case that $c_3(T_X \otimes \omega_X) \neq 0$ the normalized PT correlators
$Z_{\beta}\left( \ch_{k_1}(\gamma_1) \cdots \ch_{k_n}(\gamma_n) \right)$
are not quasi-Jacobi forms.\footnote{For example, for elliptic Calabi-Yau threefolds by the string equation we have
\[ Z_{0}(\ch_3(1)) =  
c_3(T_X \otimes \omega_X) p \frac{d}{dp} \left( \log f(p,q) \right) \]
where $f(p,q) = \prod_{\ell, m \geq 1} (1- p^{\ell} q^{m} )^{-\ell}$.
The function $p \frac{d}{dp} \log f(p,q)$ can be seen to be {\em not} a quasi-Jacobi form
(for once it doesn't converge near $z=0$ where $p=e^{z}$).}
Instead we propose here to use the theory of $\pi$-stable pairs defined in \cite{FM}, which is a modified stable pairs theory that is more efficient and adapted to the elliptic fibration structure.\footnote{This is parallel to the motivation to use Pandharipande-Thomas invariants instead of Donaldson-Thomas invariants: rationality of generating series does not hold for arbitrary normalized generating series of Donaldson-Thomas invariants in the presence of descendents of $1$.  Only the PT invariants defined by the more economical moduli space of stable pairs is expected to have rationality properties in general, see \cite{Pdesc} or \cite[Sec.5]{Oblomkov}. The intuitive reason is the interaction of the 'floating points' with the descendents of $1$. In our case considered here we want to similarly exclude interaction of 'floating 1-dimensional sheaves supported on fibers' with descendent classes.}

Let $\CC$ be the full subcategory of $\Coh(X)$ consisting of sheaves with at most $1$-dimensional support which are supported on fibers of $\pi : X \to B$.
Let
\[\BT \subset \CC \]
be the smallest extension-closed full subcategory which contains all $\mu$-semistable sheaves of positive slope
(where the slope is defined by any ample class, see Section~\ref{subsec:stability}).

\begin{defn}[\cite{FM}]
A $\pi$-stable pair on $\pi : X \to B$ is a pair $(F,s)$ consisting of
\begin{itemize}
\item a $1$-dimensional coherent sheaf $F$ such that $\Hom(T,F) = 0$ for all $T \in \BT$,
\item a section $s : \CO_X \to F$ with cokernel in $\BT$.
\end{itemize}
\end{defn}

We prove here the existence of the moduli space of $\pi$-stable pairs:
\begin{thm} The moduli functor $P^{\pi}_{n,\beta}(X)$ of $\pi$-stable pairs (defined in Section~\ref{subsec:moduli space}) is represented by a proper algebraic space and is equipped with a perfect obstruction theory.
\end{thm}

It follows that we can define $\pi$-stable pair invariants parallel to stable pair invariants.
The invariants, partition functions, and so on, of $\pi$-stable pairs are defined by a supscript $\pi\text{-}\PT$.
We conjecture the following wall-crossing between stable pair and $\pi$-stable pair invariants:
\begin{conjecture} \label{conj:wallcrossing}
If $c_3(T_X \otimes \omega_X) = 0$, then
Pandharipande-Thomas and $\pi$-stable pair invariants agree:
\[
\left\langle \ch_{k_1}(\gamma_1) \cdots \ch_{k_n}(\gamma_n) \right\rangle^{X, \PT}_{n,\beta}
=
\left\langle \ch_{k_1}(\gamma_1) \cdots \ch_{k_n}(\gamma_n) \right\rangle^{X, \pi\textup{-}\PT}_{n,\beta}
\]
\end{conjecture}
In the case $c_3(T_X \otimes \omega_X) \neq 0$, we expect that
$\pi$-stable pair and PT invariants are connected by a wall-crossing formula.
The formula should be similar in shape to Oblomkov's conjecture relating descendent Pandharipande-Thomas and Donaldson-Thomas invariants,
see \cite[Conj.5.2.1]{Oblomkov}. However, we do not have many computations of $\pi$-stable pair invariants whenever $c_3(T_X \otimes \omega_X) \neq 0$ and $X$ is not Calabi-Yau, so the exact shape of the formula is not clear yet.
Finally, we conjecture the following extension of the quasi-Jacobi form property: 
\begin{conj}[General quasi-Jacobi form conjecture] \label{conj:QJac pi-PT intro} We have
\[
Z_{\beta}^{\piPT}\left( \ch_{k_1}(\gamma_1) \cdots \ch_{k_n}(\gamma_n) \right)
=
\Delta(q)^{\frac{1}{2} c_1(N_{B/X}) \cdot \beta}
\sum_{\alpha = (\beta_1, \ldots, \beta_{\ell})}
\frac{\varphi_{\alpha}(p,q)}
{ \prod_{i=1}^{\ell} \Theta(p^{\mathrm{div}(\beta_i)}, q\big)^2 }
\]
where $\alpha$ is as in Conjecture~\ref{conj:QJac introduction} and $\varphi_{\alpha} \in \QJac$.
\end{conj}

The formulation of the holomorphic anomaly equation for $\pi$-stable pair invariants
is an interesting problem that requires further investigation.

\subsection{Comparision with DT/PT}
There is a certain analogy of the above conjectures with
the standard conjectures about the generating series of curve-counting Donaldson-Thomas invariants,
defined for $\beta \in H_2(X,\BZ)$ by
\[ \left\langle \prod_i \ch_i(\gamma_i) \right\rangle^{\DT}_{\beta} = \sum_{n\in \BZ} q^n \int_{[ \Hilb_{n,\beta}(X) ]^{\vir}} \prod_i \ch_i(\gamma_i), \]
where we integrate over the Hilbert scheme of $1$-dimensional subschemes. 

Under this analogy the normalization factor $\left\langle 1 \right\rangle^{X, \PT,\pi}_{0}$ of Conjecture~\ref{conj:normalization}
is parallel to the series of degree zero Donaldson-Thomas invariants evaluated by \cite{LP, Li} to be
\[
\left\langle 1 \right\rangle_{\beta=0}^{\DT}
 \ =\  M(-q)^{c_3(T_X \otimes \omega_X)},
\]
where $M(q) = \prod_{n \geq 1} (1-q^n)^n$ is the Mac-Mahon function. 
Moreover, if $c_3(T_X \otimes \omega_X)=0$ then
according to Conjecture 5.2.1 of \cite{Oblomkov} and \cite{PT2} the normalized DT series
\begin{equation} \label{dafd}
\frac{ \left\langle \prod_i \ch_i(\gamma_i) \right\rangle^{\DT}_{\beta} }{\left\langle \prod_i \ch_i(\gamma_i) \right\rangle^{\DT}_{0} }
\end{equation}
is a rational function satisfying a functional equation \cite{PT2}.
This is parallel to Conjecture~\ref{conj:QJac introduction},
where we conjectured under the same assumptions that the normalized correlators are quasi-Jacobi forms.
The expected equality of \eqref{dafd} with the generating series of PT invariants
is in analogy with Conjecture~\ref{conj:wallcrossing}.
The rationality of PT invariants for all threefolds is analogous to Conjecture~\ref{conj:QJac pi-PT intro}.
We summarize this discussion in the following table.

\begin{table}[h!]
\begin{center}
{\renewcommand{\arraystretch}{1.8}\begin{tabular}{| c | c | c |}
\hline
 & DT/PT theory for $X$ & PT/$\pi$-PT theory for $X \to B$ \\
\hline
Normalization factor  & 
$\left\langle 1 \right\rangle^{\DT}_{0}$ & $\left\langle 1 \right\rangle^{X,\PT,\pi}_{0}$ \\
\hline
Correspondence & DT/PT & PT/$\pi$-PT \\
\hline
If $c_3(T_X \otimes \omega_X)=0$ & Rationality of
$\frac{ \left\langle \prod_i \ch_i(\gamma_i) \right\rangle^{\DT}_{\beta} }{ \left\langle 1 \right\rangle^{\DT}_{0} }$
& Quasi-Jacobiness of
$\frac{ \left\langle \prod_i \ch_i(\gamma_i) \right\rangle^{\PT,\pi}_{\beta}}{ \left\langle 1 \right\rangle^{\PT}_{0}}$ \\ 
\hline
Any $c_3(T_X \otimes \omega_X)$
& Rationality of
$ \left\langle \prod_i \ch_i(\gamma_i) \right\rangle^{\PT}_{\beta} $
& Quasi-Jacobiness of
$ \left\langle \prod_i \ch_i(\gamma_i) \right\rangle^{\piPT,\pi}_{\beta}$ \\ 
\hline
\end{tabular}}
\label{table:virtual dimension}
\end{center}
\end{table}

\subsection{Plan of the paper}
In Section~\ref{sec:Preliminaries} we introduce some background on Pandharipande-Thomas theory, elliptic threefolds and quasi-Jacobi forms.
In Section~\ref{sec:HAE} we state with all details the $G_2$-holomorphic anomaly equation in the absolute case, and explain several consequences.
In Section~\ref{sec:First examples} and~\ref{sec:extended example K3xC} we discuss the examples
$\p^2 \times E$, $\BC^2 \times E$ and $K3 \times \BC$.
Section~\ref{sec:pi stable pairs} discusses our results on $\pi$-stable pairs.

\subsection{Acknowledgements}
The authors were supported by the starting grant 'Correspondences in enumerative geometry: Hilbert schemes, K3 surfaces and modular forms', No 101041491 of the European Research Council.

\section{Preliminaries} \label{sec:Preliminaries}
\subsection{Pandharipande-Thomas theory}
We recall the following basic fact:

\begin{lemma}(Divisor and String equation) \label{lemma:divisor_string_equation} Let $X$ be a smooth projective threefold and let $D \in H^2(X,\BQ)$.
Then with $d_{\beta} = \int_{\beta} c_1(X)$ we have
\begin{gather*}
\left\langle \ch_2(D) \ch_{k_1}(\gamma_1) \cdots \ch_{k_n}(\gamma_n) \right\rangle^{X}_{n, \beta}
=
(D \cdot \beta) \left\langle \ch_{k_1}(\gamma_1) \cdots \ch_{k_n}(\gamma_n) \right\rangle^{X}_{\beta}, \\
\left\langle \ch_3(1) \ch_{k_1}(\gamma_1) \cdots \ch_{k_n}(\gamma_n) \right\rangle^{X}_{n, \beta}
=
(n - \frac{1}{2} d_{\beta}) \left\langle \ch_{k_1}(\gamma_1) \cdots \ch_{k_n}(\gamma_n) \right\rangle^{X}_{\beta}.
\end{gather*}
\end{lemma}
\begin{proof}
In both cases $\ch_2(D)$ and $\ch_3(1)$ are cohomology classes on $P_{n,\beta}(X)$ of degree $0$.
Hence we can determine them by restricting to fibers and applying the Hirzebruch-Riemann-Roch theorem.
(For example, in the second case one uses $\ch_3(F) = \chi(F) - \frac{1}{2} d_{\beta}$.)
\end{proof}

\subsection{Elliptic threefolds}
Let $B$ be a smooth projective surface.
We assume that the threefold $X$ admits an elliptic fibration
\[ \pi : X \to B, \]
by which we mean a flat morphism with $\omega_{\pi}$ trivial on all fibers, such that:
\begin{itemize}
\item $\pi$ is equipped with a section
\[ \iota : B \to X, \quad \pi \circ \iota = \id, \]
\item $\pi : X \to B$ is a Weierstra{\ss} model \cite{D}.
\end{itemize}

\begin{rmk}
The second assumption is used in our discussion only for the following lemma.
The first assumption can probably also be weakened but would require us to work with
quasi-Jacobi forms for a congruence subgroup below. We refer to \cite{CKS} for related discussions.
\end{rmk}
%

The normal bundle of the section $\iota$ is denoted by
\[ N_{B/X} = N_{\iota} = \iota^{\ast} ( \Omega_{X/B}^{\ast} ). \]

Cohomology classes on $B$ naturally define classes on $X$ by pulling them back along~$\pi$.
We often suppress the pullback by $\pi$.
This convention is in particular followed in this lemma:

\begin{lemma} \label{lemma:Chern classes} Let $\ell = c_1(N_{B/X})$. Then
\begin{align*}
c_1(T_X) & = c_1(T_B) + \ell\\
c_2(T_X) & = c_2(T_B) + 12 \ell^2 + \ell c_1(T_B) - 12 \iota_{\ast}(\ell) \\
c_3(T_X) & = \iota_{\ast} (-72 \ell^2 - 12 \ell c_1(T_B)).
\end{align*}
\end{lemma}

\begin{proof}
This follows from the description of $X$ as a Weierstra\ss model, which we recall quickly; we refer to \cite{D} for details.
One defines the rank $3$ vector bundle
\[ \CE = \pi_{\ast} \CO_X(3 B_0), \]
where $B_0 = \iota(B)$.
By checking fiberwise one has $\pi_{\ast} \CO_X(B_0) = \CO_B$, and for all $n \geq 1$ the exact sequence
\[ 0 \to \CO_X(n B_0) \to \CO_X( (n+1)B_0 ) \to \CL^{n+1} \to 0, \]
where $\CL = N_{B/X}$. This implies that $\CE$ admits a natural filtration with graded piece
\[ \mathrm{gr}( \CE) = \CO_B \oplus \CL^{\otimes 2} \oplus \CL^{\otimes 3}. \]

The line bundle $\CO_X(3 B_0)$ is relatively very ample, and hence defines an embedding\footnote{We use the geometric convention $\p(F) = \mathrm{Proj}( \mathrm{Sym}^{\bullet} F^{\vee} )$ for a vector bundle $F$.}
\[ j : X \hookrightarrow \p( \CE^{\vee} ) \]
such that $j^{\ast}( \CO_{\p(\CE^{\vee} )}(1)) = \CO_X(3e)$.
By taking local coordinates $x$ on $\CL^2$ and $y$ on $\CL^3$ and using the cubic equation $y^2 = x^3 + \ldots$ one sees that
\[ \CO_{\p( \CE^{\vee} )}(X) = \CL^{-6} \otimes \CO_{\p( \CE^{\vee} )}(3) \]
that is the divisor $X \subset \p( \CE^{\vee} )$ is cut out by a section of 
$\CL^{-6} \otimes \CO_{\p( \CE^{\vee} )}(3)$.

The rest now follows from standard arguments.
Let $p :\p( \CE^{\vee} ) \to B$ be the projection. We use the three sequences:
\begin{gather*}
0 \to \CO_{\p( \CE^{\vee} )} \to p^{\ast}(\CE^{\vee}) \otimes \CO_{\p( \CE^{\vee} )}(1) \to T_{\p( \CE^{\vee} )/B} \to 0 \\
0 \to T_{\p( \CE^{\vee} )/B} \to T_{\p( \CE^{\vee} )} \to p^{\ast}(T_B) \to 0 \\
0 \to T_X \to T_{\p( \CE^{\vee} )}|_X \to \CO_{\p( \CE^{\vee} )}(X)|_X \to 0.
\end{gather*} 
If we let
\[ \xi = c_1( \CO_{\p( \CE^{\vee} )}(1) ) \]
and use $\xi|_{X} = 3 B_0$ and $B_0^2 = B_0 \ell$ and $c(\CE^{\vee}) = 1 - 5 \ell + 6 \ell^2$ this yields:
\begin{align*}
c(T_X) & = \frac{ c( \pi^{\ast} T_B) c( p^{\ast} \CE^{\vee} \otimes \CO(1) ) }{ c( \CL^{-6} \otimes \CO(3) ) } \\
& = \frac{ (1 + c_1(T_B) + c_2(T_B)) ( 1 + 9B_0 - 5 \ell - 3 B_0 \ell + 6 \ell^2 ) }{1 - 6 \ell + 9 B_0 }
\end{align*}
which gives the lemma as claimed.
\end{proof}

The following constant will play an important role.
\begin{cor} \label{cor:constant} We have
\[ c_3(T_X \otimes \omega_X) = c_3(X) - c_1(X) c_2(X) = -60 \int_{B} c_1(N_e)^2. \]
\end{cor}

We give the most basic computation of Pandharipande-Thomas invariants of $X$.
Let 
\[ F \in H_2(X,\BZ) \]
be the class of a fiber of $\pi : X \to B$.

\begin{lemma} \label{lemma:basic computation}
With $N = N_{B/X}$ we have
\[ \sum_{d \geq 0} q^d \blangle 1 \brangle^X_{0,dF} = \prod_{m \geq 1} (1-q^m)^{-e(B) + c_1(N) \cdot (K_B - c_1(N))}. \]
\end{lemma}

\begin{proof}
We have the isomorphism $P_0(X,dF) \cong B^{[d]}$ given by sending the ideal sheaf $I_z$ of a length $d$ subscheme $z \subset B$
to the stable pair $\pi^{\ast}(I_z) = [ \CO_X \to \CO_{\pi^{-1}(z)}]$.
The moduli space is smooth, so the invariant is the top Chern class of the obstruction bundle.
The virtual tangent bundle at a point $I = \pi^{\ast}(I_z)$ is:
\[ T_{P_0(X,dF)}^{\vir}|_{I} = R\Hom_X(I, I)_0[1] = (R\Hom_X(I,I) - R\Gamma(X,\CO_X))[1]. \]

The relative dualizing sheaf $\omega_{\pi}$ is a line bundle and trivial on fibers of $\pi$.
Restricting to the section $\iota(B)$ we find
\[ \omega_{\pi}|_{\iota(B)} = \Omega_{X/B}|_{\iota(B)} = N^{\ast} \]
where $N = N_{B/X}$. Hence we get
\[ \omega_{\pi} = \pi^{\ast} N_{B/X}^{\ast} \]
and so 
\[ R^1 \pi_{\ast} \CO_X = \pi_{\ast}( \omega_{\pi} )^{\vee} = N_{B/X}. \]
We conclude that
\[ R \Gamma(X, \CO_X) = R \Gamma(S, R \pi_{\ast} \CO_X) = R \Gamma(S, \CO_S - N ) = R \Gamma(S, \CO_S) - R \Gamma(S, N). \]
Moreover,
\[ R\Hom_X(I,I) = R\Hom_S( I_z, I_z \otimes (\CO_S - N )). \]
Thus
\[
T_{P_0(X,dF)}^{\vir}|_{I}
=
\underbrace{R \Hom_S(I_z, I_z)[1] + R\Gamma(S, \CO_S)}_{T_{S^{[d]}, I_z}} + \underbrace{R \Hom_S(I_z, I_z \otimes N) + R \Gamma(S,N)[-1]}_{Obs[-1]}. \]

We hence find
\[
\sum_{d \geq 0} q^d \int_{[ P_{0}(X,dF) ]^{\vir}} 1
=
\int_{S^{[d]}} c_{2d}( - R\Hom_S(I_z, I_z \otimes N) + R \Gamma(S,N)).
\]
The claim hence follows from \cite[Cor. 1]{CO}.
\end{proof}

In two special cases this computes all Pandharipande-Thomas invariants in class $dF$:
\begin{prop} \label{prop:PT0 check} Let $\pi : X \to B$ be one of the following two cases:
\begin{itemize}
\item[(a)] the projection $S \times E \to S$ for any surface $S$ and elliptic curve $E$,
\item[(b)] the product map $S \times \Sigma \to C \times \Sigma$ for any elliptic surface $S \to C$ and curve $\Sigma$.
\end{itemize}
Then we have
\[ \sum_{n \in \BZ} \sum_{d \geq 0} \blangle 1 \brangle^X_{n,dF} (-p)^n q^d = \prod_{m \geq 1} (1-q^m)^{-e(B) + c_1(N) \cdot (K_B - c_1(N))}. \]
In particular, Conjecture~\ref{conj:normalization} holds in these cases.
\end{prop}
\begin{proof}
For (a) if $n > 0$ then the elliptic curve $E$ acts free on $P_n(X, d [E])$ by translation.
The virtual fundamental class is then pulled back from the quotient $P_n(X,d[E])/E$, and hence its integral vanishes, see \cite{ObReduced}. In case $n<0$ the moduli space is empty, and for $n=0$ the invariants are computed by Lemma~\ref{lemma:basic computation}.

For (b) by a degeneration argument we can assume that $\Sigma = \p^1$
If $S = \p^1 \times E$, then the claim follows from part (a). Otherwise, by a further degeneration argument, where we degenerate $S$ to a chain of rational elliptic surfaces,
we can assume that $S$ is a rational elliptic surface.
Applying the GW/PT correspondence \cite{PaPix} and using the localization formula for the natural torus action on $\p^1$ then yields the vanishing for $n>0$ by a direct computation.
\end{proof}

\subsection{Quasi-Jacobi forms}
\label{subsec:quasi-Jacobi forms}
Jacobi forms are holomorphic functions $f : \BC \times \BH \to \BC$
which satisfy a transformation law under the action of the Jacobi group $\SL_2(\BZ) \ltimes \BZ^2$, see \cite{EZ}.
Quasi-Jacobi forms are holomorphic parts of almost holomorphic Jacobi forms, see \cite{Lib, elfib,vIOP} for more details.
We introduce here quasi-Jacobi forms in an adhoc fashion in terms of their generators.
In particular, we will identify a quasi-Jacobi form with its Fourier expansion in the variables
\begin{equation} p = e^{2 \pi i x}, \quad q = e^{2 \pi i \tau}, \quad (x,\tau) \in \BC \times \BH. \label{Vars} \end{equation}

For even $k \geq 2$ define the index $0$ and weight $k$ Eisenstein series
\[ G_k(q) = - \frac{B_k}{2 \cdot k} + \sum_{n \geq 1} \sum_{d|n} d^{k-1} q^n. \]
Let $D_p = p \frac{d}{dp}$,
recall $\Theta(p,q)$ from \eqref{Theta}, and consider the following functions:
\begin{gather*}
\A(p,q) = D_p \log \Theta(p,q), \quad  \quad 
\wp(p,q) = - D_p \A(p,q) - 2 G_2(q), \\
\wp'(p,q) = D_p \wp(p,q).
\end{gather*}
We let $\Theta, \A, \wp, \wp'$ be of weight $-1,1,2,3$ and index $1,0,0,0$ respectively.

\begin{defn}[{\cite[Prop.2.1]{vIOP}}]
The algebra of quasi-Jacobi forms $\QJac$ is the subring of the free polynomial algebra
$\BC[ \Theta, \A, G_2, \wp, \wp', G_4 ]$ consisting of those polynomials which define holomorphic functions $\BH\times \BC \to \BC$
under \eqref{Vars}.
The algebra is bigraded by weight and index:
\[ \QJac = \bigoplus_{m \in \frac{1}{2} \BZ_{\geq 0}} \bigoplus_{k \in \BZ} \QJac_{k,m}. \]
An element of $\QJac_{k,m}$ is called a quasi-Jacobi form of weight $k$ and index $m$.
\end{defn}

The generators $G_2$ and $\A$ are strictly quasi-Jacobi forms; the remaining generators are examples of (meromorphic) Jacobi forms.
It is hence natural to consider the formal derivative operators in the generators $G_2$ and $\A$:
\[ \frac{d}{d G_2} : \QJac(\Gamma)_{k,m} \to \QJac(\Gamma)_{k-2,m}, \quad \frac{d}{d \A} : \QJac(\Gamma)_{k,m} \to \QJac(\Gamma)_{k-1,m} \]
The operators $\frac{d}{dG_2}$, $\frac{d}{d \A}$, $D_p$ and $D_{\tau} := D_q := q \frac{d}{dq}$ act on $\QJac$.
Let~$\mathrm{wt}$ and~$\mathrm{ind}$ be the operators which act on~$\QJac_{k,m}$ by multiplication by the weight~$k$ and the index~$m$ respectively.
By \cite[(12)]{elfib} we have the commutation relations:
\begin{equation} \label{eq:comm relations 1}
\begin{alignedat}{2}
\left[ \frac{d}{dG_2}, D_{\tau} \right] & = -2 \mathrm{wt}, & \qquad \qquad \left[\frac{d}{d \A}, D_p \right] & = 2 \mathrm{ind} \\
\left[\frac{d}{dG_2}, D_p \right] & = -2\frac{d}{dA}, &\qquad  \left[\frac{d}{dA}, D_{\tau}\right] & = D_p.
\end{alignedat}
\end{equation}

The algebra of (weak) Jacobi forms $\Jac$ is the subring of $\QJac$ consisting of functions $f$ with $\frac{d}{dG_2} f = \frac{d}{d \A} f = 0$.

\section{Holomorphic anomaly equation} \label{sec:HAE}
We state with all details our main holomorphic anomaly equation,
then discuss the compatibility with the string and divisor equation and deduce some consequences.

\subsection{Generating series}
Define the divisor class
\[ W = [\iota(B)] - \frac{1}{2} \pi^{\ast}(c_1(N_{B/X})) \ \in H^2(X). \]

Recall that for any fixed class $\beta \in H_2(B,\BZ)$ we defined the series:
\[
\left\langle \ch_{k_1}(\gamma_1) \cdots \ch_{k_n}(\gamma_n) \right\rangle^{X}_{\beta}
= 
\sum_{\substack{\widetilde{\beta} \in H_2(X,\BZ) \\ \pi_{\ast} \widetilde{\beta} = \beta }}
\sum_{m \in \frac{1}{2} \BZ}
i^{2m} p^{m}  q^{W \cdot \widetilde{\beta}} \left\langle \ch_{k_1}(\gamma_1) \cdots \ch_{k_n}(\gamma_n) \right\rangle^{X}_{m+\frac{1}{2} d_{\widetilde{\beta}}, \widetilde{\beta}}
\]
where $i = \sqrt{-1}$ and $d_{\widetilde{\beta}} = \int_{\widetilde{\beta}} c_1(T_X)$.
We also use the normalization:
\[
Z_{\beta}\left( \ch_{k_1}(\gamma_1) \cdots \ch_{k_n}(\gamma_n) \right)
= 
\frac{ \left\langle \ch_{k_1}(\gamma_1) \cdots \ch_{k_n}(\gamma_n) \right\rangle^{X, \PT}_{\beta} }{\left\langle 1 \right\rangle^{X, \PT}_{0} }
\]

\subsection{Main equation}
Recall from \cite[Sec.2]{elfib} that the operator
$[W \cup (- ),\pi^*\pi_*( - )]$ acts semi-simply on $H^{\ast}(X)$ and yields a decomposition
\[ H^*(X) = H_+^*\oplus H_-^*\oplus H_{\perp}^*\]
into eigenspaces of eigenvalue $1$, $-1$, and $0$ respectively.
Below we always assume that each $\gamma_i \in H^{\ast}(X)$ lies in such an eigenspace.
We let $\wt(\gamma_i)$ denote its eigenvalue.

Define the following class in $H^{\ast}(X^3)$:
 \begin{align*} \CE & \coloneqq \Delta_{X,12} \cdot \pi^*\Delta_{B,13}+\Delta_{X,13} \cdot \pi^*\Delta_{B,12} + \Delta_{X,23} \cdot \pi^*\Delta_{B,12}  \\
& \quad -\pi^*\Delta_{B,123}\big(\pr_1^*W + \pr_2^* W + \pr_3^* W \big) \end{align*}
where we denote by $\pi$ also the map 
\[ X\times X\times X \longrightarrow B\times B \times B, \]
we let $\Delta_{X, I}$ be the class of the locus in $X^n$ where all the points labeled by
 $I \subset \{ 1, \ldots, n \}$ coincide,
and we let $\pr_I$ be the projections to the factors labeled by $I$.
We view $\CE$ here as a correspondence, i.e. for $\gamma \in H^{\ast}(X)$ and $\Gamma \in H^{\ast}(X^2)$ we define
\begin{equation} \label{CE action}
\begin{gathered} 
\CE(\gamma) = \pr_{23 \ast}( \pr_1^{\ast}(\gamma) \CE ) \\
\CE(\Gamma) = \pr_{3 \ast}( \pr_{12}^{\ast}(\Gamma) \CE ).
\end{gathered}
\end{equation}
The following lemma is straightforward:
\begin{lemma} Let $\gamma, \gamma' \in H^{\ast}(X)$ and $\alpha, \alpha' \in H^{\ast}(B)$.
\begin{itemize}
\item (Symmetry) $\CE( \gamma \boxtimes \gamma' ) = \CE( \gamma' \boxtimes \gamma )$
\item ($H^{\ast}(B)$-linearity)
\begin{gather*}
\CE( \pi^{\ast}(\alpha) \cdot \gamma) = \pr_1^{\ast}(\alpha) \cdot \CE(\gamma) \\
\CE( \pi^{\ast}(\alpha \boxtimes \alpha') \cdot \gamma \boxtimes \gamma') = (\alpha \cdot \alpha') \cdot \CE(\gamma \boxtimes \gamma')
\end{gather*}
\item $\CE(1) = \Delta_B$,\ $\CE(W) = \Delta_X$,\ $\CE(\alpha) = (\alpha_1 + \alpha_2) \Delta_{B}$ for $\alpha \in H^{\ast}_{\perp}(X)$
\item $\CE(\gamma \boxtimes 1) = \pi^{\ast} \pi_{\ast}(\gamma)$,\ $\CE(W \boxtimes \gamma) = \gamma$
\item $\CE(\gamma \boxtimes \gamma') = \pi^{\ast} \pi_{\ast}(\gamma \cdot \gamma') + \gamma \cdot \pi^{\ast} \pi_{\ast}(\gamma')
+ \pi^{\ast} \pi_{\ast}(\gamma) \cdot \gamma' 
- \pi^{\ast}( \pi_{\ast}(W \gamma) \cdot \pi_{\ast}(\gamma')) - \pi^{\ast}( \pi_{\ast}(\gamma) \cdot \pi_{\ast}( \gamma' W ) )
- \pi^{\ast}( \pi_{\ast}(\gamma) \pi_{\ast}(\gamma') ) W$
\end{itemize}
\end{lemma}

Given a class $\gamma \in H^{\ast}(X \times X)$ we will also write
\begin{multline*}
 (a-1+\wt^L)! (b-1+\wt^L)! \tch_a \tch_b( \gamma ) \\
= \sum_{i} (a-1+\wt(\gamma_i^L) )! (b-1+\wt(\gamma_i^R))! \tch_a(\gamma_i^L) \tch_b(\gamma_i^R) 
\end{multline*}
whenever $\gamma = \sum_i \gamma_i^{L} \boxtimes \gamma_i^R$ is a $\wt$-homogeneous K\"unneth decomposition of $\gamma$.

We make the following convention regarding factorials:
In all sums below we sum only over those terms where all
in the summand all occuring factorials $a!$ have $a \geq 0$.
For example,
$\sum_{k \in \BZ} \frac{x^k}{k!}$ will mean for us $k \geq 0$.
Similarly, in all appearing binomial coefficients $\binom{n}{a,b} = n!/(a! b!)$ we assume that $n,a,b \geq 0$.

Given a function $Z : H^{\ast}(X) \to \BQ$ and $\Gamma \in H^{\ast}(X^2)$
the expression
$Z(\Gamma_1) \cdot Z(\Gamma_2)$
stands for the sum
\begin{equation} Z(\Gamma_1) \cdot Z(\Gamma_2) := \sum_{i} Z( \phi_i ) Z( \phi_i^{\vee} ) \label{roijeg} \end{equation}
where $\Gamma = \sum_i \phi_i \otimes \phi_i^{\vee} \in H^{\ast}(X \times X)$ is a K\"unneth decomposition.
We will apply this below to the diagonal class $\Delta_B \in H^{\ast}(B^2)$.

We are ready to state our main conjecture:

\begin{conj} \label{conj:HAE} Assume that $c_3(T_X \otimes \omega_X) = 0$.
Then we have:
\begin{enumerate}
\item[(a)] (Quasi-Jacobi form property)
\[
Z_{\beta}\left( \ch_{k_1}(\gamma_1) \cdots \ch_{k_n}(\gamma_n) \right)
=
\Delta(q)^{\frac{1}{2} c_1(N_{B/X}) \cdot \beta}
\sum_{\alpha = (\beta_1, \ldots, \beta_{\ell})}
\frac{\varphi_{\alpha}(p,q)}
{ \prod_{i=1}^{\ell} \Theta(p^{\mathrm{div}(\beta_i)}, q\big)^2 }
\]
where $\alpha$ runs over all decompositions
$\beta = \beta_1 + \ldots \beta_k$ into effective classes $\beta_i \in H_2(B,\BZ)$
which are of divisibility $\mathrm{div}(\beta_i)$ in $H_2(B,\BZ)$,
and all $\varphi_{\alpha} \in \QJac$.\\

\item[(b)] (Holomorphic anomaly equation) We have:
\begin{small}
\begin{align*}
& \frac{d}{dG_2} Z_{\beta}\left( \tch_{k_1}(\gamma_1) \cdots \tch_{k_n}(\gamma_n) \right) = \\
& -2 \sum_{i < j} \sigma_{ij} Z_{\beta}\left( \tch_{k_i-1}(\gamma_i \Delta_{B,1}) \tch_{k_j-1}(\gamma_j \Delta_{B,2}) \prod_{\ell \neq i,j} \tch_{k_{\ell}}(\gamma_{\ell}) \right) \\
& - \sum_{i=1}^{n} \sigma_i Z_{\beta}\left( \tch_{k_i-2}( \gamma_i \cdot c_2(B)) \prod_{\ell \neq i} \tch_{k_{\ell}}(\gamma_{\ell}) \right) \\
& + \sum_{i=1}^{n} \sum_{m_1 + m_2 = k_i} \sigma_i Z_{\beta}\left( (-1)^{(1+\wt_L)(1+\wt_R)} \frac{(m_1 - 1 + \wt^L)! (m_2 - 1 + \wt^R)!}{(k_i - 2 + \wt(\gamma_i))!} \tch_{m_1} \tch_{m_2} ( \CE( K_X \cdot \gamma_i )) \prod_{\ell \neq i} \tch_{k_{\ell}}(\gamma_{\ell}) \right) \\
& -2 \sum_{i < j} \sigma_{ij} (-1)^{(1+\wt(\gamma_i)) (1+\wt(\gamma_j))} \binom{ k_i + k_j - 4 + \wt(\gamma_i) + \wt(\gamma_j) }{ k_i-2 + \wt(\gamma_i), \, k_j - 2 + \wt(\gamma_j) } 
 Z_{\beta}\left( \tch_{k_i+k_j-2}( \CE( \gamma_i \boxtimes \gamma_j )) \prod_{\ell \neq i,j} \tch_{k_{\ell}}(\gamma_{\ell}) \right),
\end{align*}
\end{small}
\noindent
where $\sigma_i$ is the sign obtained by permuting the $i$-th entry of $(\gamma_1, \ldots, \gamma_n)$ to the left-most position,
and $\sigma_{ij}$ is the sign obtained by permuting the $i$-th and $j$-th entry to the left-most position.
\end{enumerate}
\end{conj}

The above conjecture differs from the Conjectures~\ref{conj:QJac introduction} and~\ref{conj:HAE introduction} given in the introduction
in that we do not make any claims about the weight and index of the quasi-Jacobi form,
and that we do not give the holomorpic anomaly equation for $\A$.
Instead, we show below that these are formal consequences
of the divisor and string equation (Lemma~\ref{lemma:divisor_string_equation})
and the Lie algebra relations given in \eqref{eq:comm relations 1}:

\begin{prop} \label{prop:weight}
Assume that $c_3(T_X \otimes \omega_X)=0$.
If Conjectures~\ref{conj:normalization} and \ref{conj:HAE} hold, then
the quasi-Jacobi form
$Z_{\beta}( \tch_{k_1}(\gamma_1) \cdots \tch_{k_n}(\gamma_n) )$ is
of weight\footnote{The first term can be rewritten as $\int_{\iota_{\ast} \beta} K_X = - \beta \cdot (c_1(T_B) + \ell)$.}
\[ K_X \cdot \iota_{\ast}\beta + \sum_{i=1}^{n} (k_i - 1 + \wt(\gamma_i)). \]
\end{prop}

For the proof we need the following reformulation of the divisor equation:

\begin{lemma} \label{lemma: div W} Assume that $c_3(T_X \otimes \omega_X)=0$ and Conjecture~\ref{conj:normalization}. Then
\begin{multline} \label{div equation for Z beta}
Z_{\beta}( \tch_2(W) \tch_{k_1}(\gamma_1) \cdots \tch_{k_n}(\gamma_n) ) 
=
D_{\tau} Z_{\beta}( \tch_{k_1}(\gamma_1) \cdots \tch_{k_n}(\gamma_n) ) \\
+ \Big( e(B) + c_1(N_{B/X}) \cdot c_1(T_B) \Big) G_2(q) \cdot Z_{\beta}( \tch_{k_1}(\gamma_1) \cdots \tch_{k_n}(\gamma_n) )
\end{multline}
\end{lemma}
\begin{proof}
Recall that for $D \in H^2(X,\BQ)$ we have:
\[ \tch_2(D) = \ch_2(D) - \frac{1}{24} (D \cdot c_2(T_X)). \]
By Lemma~\ref{lemma:Chern classes} and Corollary~\ref{cor:constant} we find that 
\[ \tch_2(W) = \ch_2(W) - \frac{1}{24} ( e(B) + c_1(N_{B/X}) \cdot c_1(T_B)). \]
The claim now follows from the divisor equation and
the following easy identity (obtained from computing the logarithmic derivative):
\[ D_{\tau} \prod_{n \geq 1} (1-q^n)  = \left(- G_2(q) - \frac{1}{24} \right) \prod_{n \geq 1} (1-q^n).  \qedhere \]
%
\end{proof}

\begin{proof}[Proof of Proposition~\ref{prop:weight}]
We apply $\frac{d}{dG_2}$ to \eqref{div equation for Z beta} and compute both sides.
On the right hand side we use the commutation relation $[ \frac{d}{dG_2} , D_{\tau} ] = -2 \wt$.
On the left we use the holomorphic anomaly equation (Conjecture~\ref{conj:HAE}).
Then we compare both terms and solve for the weight of $Z_{\beta}( \tch_{k_1}(\gamma_1) \cdots \tch_{k_n}(\gamma_n) )$.
By a straightforward computation this yields the claim.
\end{proof}

\begin{prop}
Assume that $c_3(T_X \otimes \omega_X) = 0$ and Conjectures \ref{conj:normalization} and ~\ref{conj:HAE} hold. Then
\begin{align*}
\frac{d}{d \A} Z_{\beta}\left( \tch_{k_1}(\gamma_1) \cdots \tch_{k_n}(\gamma_n) \right)
= & 
\sum_{i = 1}^{n} \sigma_i
Z_{\beta}\left( \ \tch_{k_i-1}( \gamma_i \Delta_{B,1} ) \tch_{2}( \Delta_{B,2} ) \prod_{\ell \neq i} \tch_{k_{\ell}}(\gamma_{\ell})  \ \right) \\
& +\sum_{i=1}^{n} \sigma_i Z_{\beta} \left( \ \tch_{k_i+1}( \pi^{\ast} \pi_{\ast}(\gamma_i) ) \prod_{\ell \neq i} \tch_{k_{\ell}}(\gamma_{\ell}) \ \right).
\end{align*}
Moreover, $Z_{\beta}\left( \tch_{k_1}(\gamma_1) \cdots \tch_{k_n}(\gamma_n) \right)$ is a quasi-Jacobi form
of index
\[
\frac{1}{2} \beta \cdot (\beta + c_1(N_{B/X})).
\]
\end{prop}
\begin{proof}
We use the string equation (Lemma~\ref{lemma:divisor_string_equation}):
\begin{equation} \label{string equation}
Z_{\beta}\left( \tch_3(1) \tch_{k_1}(\gamma_1) \cdots \tch_{k_n}(\gamma_n) \right) = p \frac{d}{dp} Z_{\beta}\left( \tch_{k_1}(\gamma_1) \cdots \tch_{k_n}(\gamma_n) \right).
\end{equation}
Applying $\frac{d}{dG_2}$ to both sides, using the Lie algebra relation
$\left[ \frac{d}{dG_2}, p \frac{d}{dp} \right] = -2 \frac{d}{d \A}$ and then the conjectural $G_2$-holomorphic anomaly equation,
implies the first claim.
For the second claim, start again with \eqref{string equation}, apply $\frac{d}{d\A}$ and use $[ \frac{d}{d\A}, p \frac{d}{dp} ] = 2 \ind$.
\end{proof}

\subsection{Further checks}
For all $\gamma \in H^{\ast}(X)$ and $\lambda \in H^2(B)$ 
the following equalities follow from the definition or the divisor equation:
\begin{align*}
 Z_{\beta}\left( \tch_0(\gamma) \tch_{k_1}(\gamma_1) \cdots \tch_{k_n}(\gamma_n) \right) & = \left( -\int_X \gamma \right) Z_{\beta}\left( \tch_{k_1}(\gamma_1) \cdots \tch_{k_n}(\gamma_n) \right) \\
Z_{\beta}\left( \tch_1(\gamma) \tch_{k_1}(\gamma_1) \cdots \tch_{k_n}(\gamma_n) \right) & = 0 \\
 Z_{\beta}\left( \tch_2(\lambda) \tch_{k_1}(\gamma_1) \cdots \tch_{k_n}(\gamma_n) \right)
& =
(\lambda \cdot (\beta + \frac{c_1(N_{e})}{2} )) Z_{\beta}\left( \tch_{k_1}(\gamma_1) \cdots \tch_{k_n}(\gamma_n) \right).
\end{align*}
A short but non-trivial computation shows that Conjecture~\ref{conj:HAE} is compatible with the above equations,
in the sense that if we apply $\frac{d}{dG_2}$ to both sides
and use the holomorphic anomaly equation, then we obtain the same result on both sides.

\section{First examples}
\label{sec:First examples}

\subsection{Calabi-Yau threefolds}
If $X$ is a Calabi-Yau threefold, 
then\footnote{Consider the sequence $0 \to T_{B} \to T_X|_{B} \to N_{B/X} \to 0$ and take the determinant.}
 $N_{B/X} = \omega_B$.
Hence the above conjectures imply that the series
\[ Z_{\beta} := Z_{\beta}(1) \]
is a (meromorphic) Jacobi-Form of weight $0$ and index $\frac{1}{2} \beta ( \beta + K_B )$.
Therefore Conjecture~\ref{conj:HAE} specializes to the Huang-Katz-Klemm conjecture of \cite{HKK}.

\subsection{The product $\p^2 \times E$}
\label{subsec:P2xE}
The following example was considered and computed in the second author's master thesis \cite{Schimpf}.
It is very different in nature from the Calabi-Yau case and hence constitutes an important piece of evidence for our conjecture.

Let $E$ be an elliptic curve and consider the trivial elliptic fibration
\[ X = \p^2 \times E, \quad \pi : X \to \p^2. \]
Consider the class of a line in $\p^2$ and of a point in $E$, respectively,
\[ H \in H^2(\p^2,\BZ), \quad \pt \in H^2(E,\BZ). \]
%

\begin{thm}[{\cite{Schimpf}}] We have
\begin{align*}
	Z_1\left( \ch_2( H^2\pt)\ch_2(H^2)\right) & = i\Theta\\
	Z_1\left( \ch_2( H^2\pt)\ch_2( H\pt)\right) & = 3iD_{\tau} \Theta \\
	Z_1\left( \ch_2( H\pt)^2 \ch_2(H^2)\right) & = 4iD_{\tau}\Theta \\
	Z_1\left( \ch_2( H\pt)^3\right) & = 3iD^2_{\tau}\Theta + 9i\frac{\left(D_{\tau}\Theta\right)^2}{\Theta} \\
	Z_1\left( \ch_3(H^2\pt )\right) & = iD_{z}\Theta = i\Theta \A \\
    Z_1\left( \ch_2(H^2 \pt) \ch_2(H \alpha) \ch_2(H \beta) \right) & = i D_{\tau}(\Theta) \\
	Z_2\left( \ch_2( H^2\pt)\ch_2(H^2)^4\right) & = \Theta^4\\
	Z_2\left( \ch_2(H^2\pt)^3\right)
	& = 3\Theta^3 D^2_z D_\tau  \Theta+3\Theta^2\left(D^2_z\Theta\right)D_\tau \Theta \\
& \quad -6\Theta^2\left( D_p D_\tau\Theta\right) D_p \Theta+3\Theta^2\left(D_\tau \Theta\right)^2 \\
	Z_2\left( \ch_2( H^2\pt)^2\ch_3(H^2)\right) & = 2\Theta^3 D_pD_\tau\Theta 
\end{align*}
\end{thm}
One checks directly that in all these cases Conjecture~\ref{conj:HAE} holds.

\subsection{The product $\BC^2\times E$} \label{ex:C2xE}
Consider the $\BG_m^2$ action on $\BC^2$ with tangent weights $t_1,t_2$ at the origin.
Consider the trivial elliptic fibration
\[ X = \BC^2\times E, \quad \pi : X \to \BC^2. \]
We have
\[ c(T_X) = (1+t_1) (1+t_2) = 1 + (t_1 + t_2) + t_1 t_2 \]
which shows that
\[ c_3(T_X \otimes \omega_X) = - (t_1 + t_2) t_1 t_2 \in H^{\ast}_{\BG_m^2}(X). \]
Hence the \emph{class} $c_3(T_X \otimes \omega_X)$ is non-zero but vanishes after restriction to $t_1=-t_2$.


Consider the equivariant modified descendent classes:
\[
\sum_{k \geq 0} \tch_k(\gamma) x^k = \frac{1}{S(x/\theta)} \sum_{k = 0}x^k \ch_k(\gamma)
\]
where $S(x) = (e^{x/2} - e^{-x/2})/x$ and $\theta^{-2} = -c_2(T_X)$. Note that this agrees with our conventions from the beginning of the paper in case the classes are non-equivariant. A localization argument shows the following:

\begin{prop}
Let $\gamma_1, \ldots, \gamma_n \in H^{\ast}(E)$. For $t_1=-t_2$, we have:
\begin{multline*}
\sum_{k_1, \ldots , k_n \geq 0}
Z^{\BC^2\times E}( \tch_{k_1}(\gamma_1) \cdots \tch_{k_n}(\gamma_n) ) x_1^{k_1} \cdots x_n^{k_n} \\
= \left( \int_{E} \gamma_1 \right) \cdots \left( \int_E \gamma_n \right)
x_1 \cdots x_n t_1^{-n}F_n(t_1x_1, \ldots, t_1x_n),
\end{multline*}
where we used the Bloch-Okounkov $n$-point correlation function \cite{character}:
\[
    F_n(x_1,\ldots,x_n) = \prod_{l\geq 1}(1-q^n) \sum_{\lambda}q^{|\lambda|}\prod_{l=1}^n \sum_{i\geq 1} e^{(\lambda_i-i+\frac{1}{2})x_l}.
\]
Moreover, Conjectures \ref{conj:QJac introduction} and \ref{conj:HAE introduction} hold in this case
\end{prop}

\begin{proof}
Assume that $t_1=-t_2$.
The moduli space of stable pairs $P_{\chi,\beta}(X)$ is empty whenever $\chi<0$.
For $\chi>0$ the natural action by $E$ on $P_{\chi,\beta}(X)$ by translation is free.
Hence taking derivative, we obtain a global non-zero vector field on $P_{\chi,\beta}(X)$,
whose Serre-dual defines a non-zero cosection of the moduli space.
This implies that the virtual class of the moduli space (and hence all invariants) in case $\chi>0$ vanish, see \cite[Section 4.3]{PP_Jap} for more details on this argument.
For $\chi=0$, as explained in the proof of Lemma \ref{lemma:basic computation}
we have the isomorphism $P_{0,d[E]}(X) = \Hilb^d(\BC^2)$ and
\[
    Z^{\BC^2\times E}( \ch_{k_1}(\gamma_1) \ldots \ch_{k_n}(\gamma_n) ) = \prod_{n\geq 1}(1-q^n) \int_{\Hilb^d(\BC^2)} e(\T_{\Hilb^d(\BC^2)}) \ch_{k_1}(\gamma_1) \ldots \ch_{k_n}(\gamma_n) ).
\]
%
%
A $\BG_m^2$-fixed stable pair corresponding to a partition $\lambda$ is of the form $[\CO \to (\pi_1)^*\CO_{\lambda}]\in D^b(\BC^2\times E)$ with $\CO_\lambda = \bigoplus_{(i,j)\in \lambda} \CO_E T_1^i T_2^j$ where we sum over all the squares in the Young diagram of $\lambda$. The contribution of an insertion is thus:
\begin{align*}
    \sum_{k\geq 1}\ch_k(\gamma) x^k &= \sum_{k\geq 1}\left(\int_{\BC^2\times E}\ch_k(\pi_1^*\CO_{\lambda})\cup\pi_2^*(\gamma)\right) x^k\\
    &= \left(\int_E \gamma \right) \frac{(1-e^{-t_1 x})(1-e^{-t_2 x})}{t_1 t_2} \sum_{(i,j)\in \lambda}e^{(-it_1-jt_2)x}
\end{align*}
and 
\[ \ch_0(\gamma) = -\frac{1}{t_1 t_2} \int_E \gamma .\]
This yields:
\begin{align}
\notag     &\sum_{k_1, \ldots , k_n \geq 0} 
Z^{E \times \BC^2}( \tch_{k_1}(\gamma_1) \cdots \tch_{k_n}(\gamma_n) ) x_1^{k_1} \cdots x_n^{k_n} \\
\notag &= \left( \int_{E} \gamma_1 \right) \cdots \left( \int_E \gamma_n \right)\prod_{n\geq 1} (1-q^n) \prod_{l=1}^n \frac{1}{t_1 t_2 S(i\sqrt{t_1 t_2} x_l)} \\
\notag &\times\sum_{\lambda} q^{|\lambda|} \prod_{l=1}^n \left( (1-e^{-x_l t_1})(1-e^{-x_l t_2})\sum_{(i,j)\in\lambda} e^{(-it_1-jt_2)x_l}-1\right) \\
\notag & = \left( \int_{E} \gamma_1 \right) \cdots \left( \int_E \gamma_n \right)\prod_{n\geq 1} (1-q^n)\prod_{l=1}^n \frac{1}{t_1t_2S(i\sqrt{t_1t_2}x_l)}\\
\notag &\times\sum_{\lambda}q^{|\lambda|} \prod_{l=1}^n \left( (1-e^{t_1x_l})\sum_{i\geq 1} \left(e^{(-\lambda_it_2-it_1)x_l}-e^{-it_1x_l})\right) - 1 \right)\\
\label{4ygfg} & = \left( \int_{E} \gamma_1 \right) \cdots \left( \int_E \gamma_n \right)\prod_{n\geq 1} (1-q^n)\prod_{l=1}^n \frac{1-e^{x_lt_1}}{t_1t_2S(i\sqrt{t_1t_2}x_l)}\\
\notag &\times\sum_{\lambda}q^{|\lambda|} \prod_{l=1}^n \sum_{i\geq 1} e^{(-\lambda_it_2-it_1)x_l}
\end{align}
In case $t_1=-t_2$, this is:
\[
    \left( \int_{E} \gamma_1 \right) \cdots \left( \int_E \gamma_n \right)\prod_{n\geq 1} (1-q^n)\prod_{l=1}^n \frac{x_l}{t_1}
    \sum_{\lambda}q^{|\lambda|} \prod_{l=1}^n \sum_{i\geq 1} e^{(\lambda_i-i+\frac{1}{2})t_1x_l},
\]
which shows the first part.

For the second part, it was shown in \cite{character}
that the $x$-coefficients of the $n$-point  correlation function
are quasi-modular forms (and hence index $0$ quasi-Jacobi forms).
Our holomorphic anomaly equation corresponds to the identity \cite[Lem.~4.2.2]{pixtonthesis}:
\begin{align*}
    \frac{d}{dG_2}F_n(x_1,\cdots,x_n) &= (x_1+\cdots+x_n)^2 F_n(x_1,\cdots,x_n) \\
    &-2 \sum_{1\leq i<j\leq n} (x_i+x_j) F_{n-1}(x_i+x_j,x_1,\cdots,\hat{x_i},\cdots,\hat{x_j},\cdots,x_n),
\end{align*}
where we take the $G_2$-derivatives factorwise.
\end{proof}

\section{Extended example: The product $K3 \times \BC$}
\label{sec:extended example K3xC}
The goal of this section is to use the holomorphic anomaly equation
to find conjectural formulas for the stationary Pandharipande-Thomas theory of
$S \times \BC$, where $S$ is a K3 surface.
The final formulas are given in Sections~\ref{subsec:general formula} and~\ref{subsec:ABC}.

\subsection{Definition and holomorphic anomaly equations}
Let $S$ be an algebraic K3 surface.
We let $\BG_m$ act on $\BC$ with weight $t$ on the tangent space at $0$
and consider the equivariant Pandharipande-Thomas theory of 
\[ X=S \times \BC. \] 
The inclusion of the fiber over $0 \in \BC$ is denoted by
\[ \iota : S \hookrightarrow X. \]
We will usually identify curve classes on $S$ with their pushforward by $\iota$.

Because of the existence of a symplectic form on $S$,
it is well-known that the standard virtual class
on $P_{n,\beta}(X)$ vanishes for $\beta \in H_2(S,\BZ)$ non-zero, see \cite{MPT}.
Instead Pandharipande-Thomas invariants of $X$ are defined by a reduced virtual class of dimension $1$,
\[ [ P_{n,\beta}(X) ]^{\red} \in A_{\BG_m, 1}( P_{n,\beta}(X) ). \]
We will use the notation:
\begin{equation} \label{reduced invariants}
\left\langle \ch_{k_1}(\gamma_1) \cdots \ch_{k_n}(\gamma_n) \right\rangle^{X, \PT,\red}_{n,\beta}
=
 \int_{[ P_{n,\beta} (X) ]^{\red}} \prod_i \ch_{k_i}(\gamma_i) \ \ \in \BQ(t)
\end{equation}
where the integral on the right stands for an application of the virtual localization formula.

\begin{rmk}(\textbf{Convention on equivariant parameter}) \label{rmk:convention}
If the $\gamma_i$ are homogeneous of complex degree $\deg(\gamma_i)$, that is $\gamma_i \in H^{2 \deg(\gamma_i)}(S)$ then we have
\[
\left\langle \ch_{k_1}(\gamma_1) \cdots \ch_{k_n}(\gamma_n) \right\rangle^{X, \PT,\red}_{n,\beta}
= c \cdot t^{-1 + \sum_k (k_i + \deg(\gamma_i) - 3)}
\]
for some $c \in \BQ$. 
Since there is no additional information contained in the exponent of $t$,
we hence usually set 
$t=1$ below.
\end{rmk}


Let $S \to \p^1$ be an elliptic K3 surface with section $B$ and fiber class $F$. We take $W=B+F$.
The threefold $X$ becomes elliptically fibered via the projection:
\[ \pi: X \to \p^1 \times \BC. \]
By \cite{QuasiK3} the invariants \eqref{reduced invariants} for imprimitive $\beta$ are determined by those where $\beta$ is primitive through a multiple cover formula.
By deformation invariance it is hence enough to consider the invariants \eqref{reduced invariants}  for the classes $\beta = B+dF$. We thus define the generating series
\begin{equation*}
 \left\langle \ch_{k_1}(\gamma_1) \cdots \ch_{k_n}(\gamma_n) \right\rangle^{\PT}
:=
\sum_{d \geq 0} \sum_{n \in \BZ} (-p)^n q^{d-1} \left\langle \ch_{k_1}(\gamma_1) \cdots \ch_{k_n}(\gamma_n) \right\rangle^{\PT,X}_{n, B + d F}. 
\end{equation*}

The Chern classes of the tangent bundle are
\[ c(T_X) = (1+t)(1+c_2(S)) = 1 + t + c_2(S) + t c_2(S). \]
In particular, $K_X = -t$ and
\[ c_3(T_X \otimes \omega_X) = 0 \in H_{\BG_m}^{\ast}(X). \]

This means that we should expect that Conjecture~\ref{conj:HAE} holds for $S \times \BC$,
as soon as we account for using the reduced virtual class
by introducing some modifications.
The modifications needed are discussed in \cite[Sec.7]{elfib}.
This comes out as follows:

\begin{conj} $\langle \tch_{k_1}(\gamma_1) \cdots \tch_{k_n}(\gamma_n) \rangle^{\PT}$
is a quasi-Jacobi form of index $-1$ and weight $-10 + \sum_i (k_i - 1 + \wt(\gamma_i))$,
of the form
\[ \left\langle \tch_{k_1}(\gamma_1) \cdots \tch_{k_n}(\gamma_n) \right\rangle^{\PT} = \frac{\varphi(p,q)}{\Theta^2(p,q) \Delta(q)} \]
for some $\varphi \in \QJac_{\ast,0}$.
\end{conj}
Here the weight is explicitly computed by
\[
\wt(\gamma) =
\begin{cases}
1 & \text{ if } \gamma \in \{ \pt, W \} \\
-1 & \text{ if } \gamma \in \{ 1, F \} \\
0 & \text{ if } \gamma \perp \{ \pt, 1, W, F \}.
\end{cases}
\]

We also expect the holomorphic anomaly equations:
\begin{conj} \label{conj:HAE K3xC}
Let $\mathsf{B} = \p^1 \times \BC$ be the base of the elliptic fibration $X \to \p^1 \times \BC$.
	\begin{align*}
		\frac{d}{d A} \left\langle \tch_{k_1}(\gamma_1) \cdots \tch_{k_n}(\gamma_n) \right\rangle^{\PT}
		= & 
		\sum_{i = 1}^{n} \sigma_i
		\left\langle \ \tch_{k_i-1}( \gamma_i \Delta_{\mathsf{B},1} ) \tch_{2}( \Delta_{\mathsf{B},2} )  \prod_{\ell \neq i} \tch_{k_{\ell}}(\gamma_{\ell}) \ \right\rangle^{\PT} \\
		& +\sum_{i=1}^{n} \sigma_i \left\langle \ \tch_{k_i+1}( \pi^{\ast} \pi_{\ast}(\gamma_i) ) \prod_{\ell \neq i} \tch_{k_{\ell}}(\gamma_{\ell}) \ \right\rangle^{\PT}.
	\end{align*}
and
		\begin{small}
	\begin{align*}
		& \frac{d}{dG_2} \left\langle \tch_{k_1}(\gamma_1) \cdots \tch_{k_n}(\gamma_n) \right\rangle^{\PT} = \\
		& -2 \sum_{i < j} \left\langle \tch_{k_i-1}(\gamma_i \Delta_{\mathsf{B},1})  \tch_{k_j-1}(\gamma_j \Delta_{\mathsf{B},2}) \prod_{\ell \neq i,j} \tch_{k_{\ell}}(\gamma_{\ell}) \right\rangle^{\PT} \\
		& - \sum_{i=1}^{n} \left\langle \tch_{k_i-2}( \gamma_i \cdot c_2(\mathsf{B})) \prod_{\ell \neq i} \tch_{k_{\ell}}(\gamma_{\ell}) \right\rangle^{\PT} \\
		& + \sum_{\substack{i=1, \ldots, n \\ m_1 + m_2 = k_i}} \frac{(-1)^{(1+\mathrm{wt}_L)(1+\mathrm{wt}_R)} (m_1 - 1 + \mathrm{wt}^L)! (m_2 - 1 + \mathrm{wt}^R)!}{(k_i - 2 + \mathrm{wt}(\gamma_i))!} \left\langle \tch_{m_1} \tch_{m_2} ( \mathcal{E}( K_X \gamma_i )) 
		\prod_{\ell \neq i} \tch_{k_{\ell}}(\gamma_{\ell}) \right\rangle^{\PT} \\
		& -2 \sum_{i < j} (-1)^{(1+\wt(\gamma_i))(1 + \wt(\gamma_j))} \binom{ k_i + k_j - 4 + \mathrm{wt}(\gamma_i) + \mathrm{wt}(\gamma_j) }{ k_i-2 + \mathrm{wt}(\gamma_i), \ k_j - 2 + \mathrm{wt}(\gamma_j) } 
		\left\langle \tch_{k_i+k_j-2}( \mathcal{E}( \gamma_i \boxtimes \gamma_j )) \prod_{\ell \neq i,j} \tch_{k_{\ell}}(\gamma_{\ell}) \right\rangle^{\PT} \\
& - \sum_{a,b} (g^{-1})_{ab} \left\langle T_{e_a} T_{e_b} \left( \tch_{k_1}(\gamma_1) \cdots \tch_{k_n}(\gamma_n) \right) \right\rangle^{\PT}
	\end{align*}
\end{small}
where
\begin{itemize}
\item the $e_a$ form a basis of $\{ F, B \}^{\perp} \subset H^2(S,\BQ)$ and
$g_{ab} = \langle e_a, e_b \rangle$ is the pairing matrix,
\item where for any $\alpha \in \{ W, F \}^{\perp} \subset V  $ 
we let $T_{\alpha}$ act on the descendent algebra 
as a derivation with action on the generators
$T_{\alpha} \ch_k(\gamma) 
= \ch_k( (F \cdot \gamma) \alpha - (\alpha \cdot \gamma) F)$.
In particular, if we let $\sigma = \sum_{a,b} (g^{-1})_{ab}  T_{e_a} T_{e_b}$ and $\alpha' \in \{ W, F \}^{\perp}$ we have
\begin{gather*}
\sigma( \ch_k(\gamma)) = -20 (\gamma \cdot F) \ch_k(F) \\
\sigma( \ch_k(W) \ch_{\ell}(W)) = 2 \sum_{a,b} g^{-1} \ch_k(e_a) \ch_{\ell}(e_b) - 20 \ch_k(F) \ch_{\ell}(W) - 20 \ch_k(W) \ch_{\ell}(F) \\
\sigma( \ch_k(W) \ch_{\ell}(\alpha) ) = -2 \ch_k(\alpha) \ch_{\ell}(F) - 20 \ch_k(F) \ch_{\ell}(\alpha) \\
\sigma( \ch_k(\alpha) \ch_{\ell}(W) ) = -2 \ch_k(F) \ch_{\ell}(\alpha) - 20 \ch_k(\alpha) \ch_{\ell}(F) \\
\sigma( \ch_k(\alpha) \ch_{\ell}(\alpha') ) = 2 (\alpha \cdot \alpha') \ch_k(F) \ch_{\ell}(F)
\end{gather*}
\item The class of the diagonal in $H_{\BG_m}(\BC \times \BC)$ 
with respect to the diagonal $\BG_m$-action is
$\Delta_{\BC} = \frac{1}{t} [ (0,0) ]$.
Hence we have
\[ \Delta_{\mathsf{B}} = \frac{1}{t} [(0,0)] (F_1 + F_2) \]
where $F_i = \pr_i^{\ast}(F)$. 
Moreover, $c_2(\mathsf{B}) = 2 t F$.
Also recall that $K_X = -t$.
\end{itemize}
\end{conj}

The most basic example is given by the Katz-Klemm-Vafa formula \cite{MPT}:
\[
\left\langle 1 \right\rangle^{\PT}
=
\frac{-1}{\Theta^2(p,q) \Delta(q)}.
\]

In the remainder of this section we will use Conjecture~\ref{conj:HAE K3xC}
to derive conjectural formulas
for 
\[ \left\langle \ch_{k_1}(\gamma_1) \cdots \ch_{k_n}(\gamma_n) \right\rangle^{\PT} \]
in the stationary case, that is whenever $\deg_{\BC}(\gamma_i)>0$ for all $i$.
This is done in several steps and uses the following two methods:
\begin{enumerate}
\item[(i)] Direct evaluations by applying the localization formula (Section~\ref{subsec:evaluation by localization})
\item[(ii)] GW/PT correspondence \cite{PaPix} (proven for $K3 \times \BC$ for primitive $\beta$ in \cite{Marked}) and computations on the Gromov-Witten side.
\end{enumerate}
The end result is given in Sections~\ref{subsec:general formula} and~\ref{subsec:ABC}.

\subsection{Evaluation by localization} \label{subsec:evaluation by localization}
Given a primitive effective class $\beta \in H_2(S,\BZ)$,
by deformation invariance we may assume that $\beta$ is irreducible.
In this case the moduli space of stable pairs satisfies:
\[ P_{n,\iota_{\ast} \beta}(X) = P_{n,\beta}(S) \times \BC. \]
Moreover, $P_{n,\beta}(S)$ is smooth of dimension $\beta^2 + n + 1$, see \cite{MPT} for references.
Let
\[ \iota : P_n(S,\beta) \times 0 \to P_n(X,\beta) \]
denote the inclusion of the zero fiber.

\begin{prop}[{\cite[Sec.1.2]{MPT}}] \label{prop:virtual class}
We have
\[ [ P_n(X,\iota_{\ast} \beta) ]^{\text{red}} 
= \iota_{\ast}\left( \frac{e(\Omega_{P_n(S,\beta)} \otimes \Ft)}{t} [P_n(S,\beta)] \right) \]
\end{prop}
\begin{proof}
This can be seen in two different ways:
Following \cite[Sec. 1.2]{MPT}, by a Serre duality computation the obstruction bundle is
\[ \mathrm{Ob} = (\Omega_{P_n(S,\beta)} \oplus -\Ft) \otimes \Ft \]
and the reduced obstruction bundle is 
\[ \mathrm{Ob}^{\text{red}} = \Omega_{P_n(S,\beta)} \otimes \Ft. \]
We obtain:
\begin{align*}
 [ P_n(X,\iota_{\ast} \beta) ]^{\text{red}} 
& = e(\Omega_{P_n(S,\beta)} \otimes \Ft) \cdot [ P_n(X, \iota_{\ast} \beta) ] \\
& = \iota_{\ast}\left( \frac{e(\Omega_{P_n(S,\beta)} \otimes \Ft)}{t} [P_n(S,\beta)] \right)
\end{align*}

Alternatively, one can apply the virtual localization formula directly. As discussed in \cite{PT} one obtains
\begin{align*}
[ P_n(X,\iota_{\ast} \beta) ]^{\text{red}} 
& =  \iota_{\ast} \left( e(E_S \otimes \Ft) \cdot [ P_n(S,\beta) ]  \right).
\end{align*}
where
\begin{itemize}
\item $E_S = R \hom_{\pi}( \CI_S, \mathbf{F})^{\vee}$,
\item $\pi : P_n(S,\beta) \times S \to P_n(S,\beta)$ is the projection,
\item $\CI_S = [ \CO \to \mathbf{F}]$ is the universal stable pair.
\end{itemize}
Note that the rank of $E_S$ is
\[ \mathrm{rank}(E_S) = \chi(I_S,F) = \chi(\CO_X - F,F) = \chi(F) - \chi(F,F) = n + \beta^2. \]
Since $\Hom_{\pi}(I_S,F)$ is the tangent space of $P_n(S,\beta)$, which is smooth of dimension $\beta^2+n+1$,
we have that $\ext^1_{\pi}(I_S,F)$ is rank $1$. In fact, $\ext^1_{\pi}(I_S,F)$ is trivial which follows from applying $\Hom(-, F)$ to the distinguished triangle $I_S \to \CO \to F$.
Hence
\[ E_S = \Tan_{P_n(S,\beta)}^{\vee} - \CO \]
and we get the previous result:
\[
 e(E_S \otimes \Ft) = e( (\Tan^{\vee} - \CO) \otimes \Ft ) = \frac{ e( \Tan^{\vee} \otimes \Ft ) }{e( \CO \otimes \Ft )} = \frac{1}{t} e(\Tan^{\vee} \otimes \Ft) \qedhere
\]
\end{proof}

For $\gamma \in H^{\ast}(S)$, consider the descendent classes on $P_n(X,\beta)$ and $P_n(S,\beta)$:
\begin{gather*}
\ch_k^X(\gamma) = \pi_{P_{n}(X,\beta), \ast}( \ch_{k}(\mathbb{F}_X - \CO) \pi_S^{\ast}(\gamma) ) \\
\ch_k^S(\gamma) = \pi_{P_{n}(S,\beta), \ast}( \ch_{k}(\mathbb{F}_S - \CO) \pi_S^{\ast}(\gamma) ).
\end{gather*}
where $(\BF_X,s_X)$ is the universal stable pair on $P_n(X,\beta) \times X$,
and $(\BF_S,s_S)$ is the universal stable pair on $P_n(S,\beta) \times S$.

\begin{lemma}  \label{lemma:descendents}
For any $\gamma \in H^{\ast}(S)$ and $k>0$ we have that
\[
\iota^{\ast} \ch_k^X(\gamma) = 
\sum_{\ell=1}^{k-1} \ch_{\ell}^S(\gamma) \frac{(-t)^{k-1-\ell} }{(k-\ell)!}.
\]
\end{lemma}
\begin{proof}
Consider the diagram
\[
\begin{tikzcd}
P_n(S,\beta) \times S \ar{d}{j} & \\
P_n(S,\beta) \times X \ar{r}{\tilde{\iota}} \ar{d}{\tilde{\pi}} & P_n(X,\beta) \times X \ar{d}{\pi} \\
P_n(S,\beta) \ar{r}{\iota} & P_n(X,\beta)
\end{tikzcd}
\]
where the bottom square is fibered.
Observe that
\begin{align*}
\ch( j_{\ast} \BF) & = j_{\ast}\left(\frac{ \ch(\BF) }{\td_{\BC}} \right) \\
& = j_{\ast}\left( \ch(\BF) \frac{1-e^{-t}}{t} \right) \\
& = j_{\ast}\left( \ch(\BF) \sum_{k \geq 0} \frac{ (-t)^{k} }{(k+1)!} \right).
\end{align*}
With $k>0$ we therefore get:
\begin{align*}
\iota^{\ast} \ch_k(\gamma)
& = \iota^{\ast} \pi_{\ast}( \ch_k(\BF_X) \pi_S^{\ast}(\gamma) ) \\
& = \tilde{\pi}_{\ast}( \widetilde{\iota}^{\ast}( \ch_k(\BF_X) ) \cdot \pi_S^{\ast}(\gamma) ) \\
& = \tilde{\pi}_{\ast}( \ch_k( j_{\ast} \BF_S ) \cdot \pi_S^{\ast}(\gamma) ) \\
& = \tilde{\pi}_{\ast}\left( \left[ j_{\ast}\left(\frac{ \ch(\BF) }{\td_{\BC}} \right) \right]_{k} \pi_S^{\ast}(\gamma) \right) \\
& = (\tilde{\pi} \circ j)_{\ast}\left( \left[ \frac{ \ch(\BF) }{\td_{\BC}} \right]_{k-1} \cdot \pi_S^{\ast}(\gamma) \right) \\
& = (\tilde{\pi} \circ j)_{\ast}\left( \sum_{\ell = 1}^{k-1} \ch_{\ell}(\BF_S) \frac{(-t)^{k-1-\ell}}{ (k-\ell)!}  \pi_S^{\ast}(\gamma) \right) \\
& = \sum_{\ell = 1}^{k-1} \ch_{\ell}^S(\gamma) \frac{(-t)^{k-1-\ell}}{ (k-\ell)!} 
\end{align*}
\end{proof}

Using the above results we can compute the following basic invariant:

\begin{prop} \label{prop:leading term}
Let $\beta \in H_2(S,\BZ)$ be an effective algebraic class with $\beta^2=-2$,
and let $F \in H^2(S,\BQ)$ with $\langle \beta,F \rangle = 1$. Then
\[ 
\sum_{n} \left\langle \ch_{2+k}(F) \right\rangle^{X, \PT,\red}_{n,\beta} (-p)^n
=
-\frac{p}{(1-p)^2} \left( \frac{(-1)^k}{(k+1)!} \frac{ 1-p^{k+1}}{ (1-p)^{k+1} } \right) t^{k-1}.
\]
\end{prop}
\begin{proof}
We may assume that $\beta$ is irreducible, and hence the class of a $(-2)$-curve $\p^1 \subset S$.
By Proposition~\ref{prop:virtual class} and Lemma~\ref{lemma:descendents}
we obtain:
\begin{align}
\int_{[ P_{n}(X,\beta) ]^{\red}} \tch^X_{2+k}(F)
& = \int_{ P_{n}(S,\beta) } \sum_{\ell=1}^{k+1} \ch_{\ell}^S( F ) \frac{ (-t)^{k+1-\ell}}{ (k+2-\ell)!} \frac{e( \Omega \otimes \Ft )}{t}. \label{eq1}
\end{align}
The moduli space $P_{n}(S,\beta)$ is isomorphic to $\p^{n-1}$, parametrizing $n-1$ points on the $(-2)$-curve.
The universal stable pair $\BF_S$ on $\p^{n-1} \times S$ is the pushforward of
$\pr_1^{\ast} \CO_{\p^{n-1}}(1) \otimes \pr_2^{\ast} \CO_{\p^1}(n-1)$ under the natural inclusion $\p^{n-1} \times \p^1 \hookrightarrow \p^{n-1} \times S$.
Let 
\[ \xi = c_1(\CO_{\p^{n-1}}(1)). \]
We obtain
\[ \ch^S_{\ell}(F) = \frac{\xi^{\ell-1}}{ (\ell-1)!}. \]
Note also that $c(T_{\p^{n-1}}) = (1+\xi)^n$ hence $c(\Omega_{\p^{n-1}}) = (1-\xi)^n$, and therefore
\[ \int_{\p^{n-1}} e(\Omega_{\p} \otimes \Ft) \xi^{\ell} 
= \mathrm{Coeff}_{\xi^{n-1-\ell}} \frac{(t-\xi)^n}{t} 
= (-1)^{n-1-\ell} t^{\ell} \binom{n}{\ell + 1}.
\]

Inserting into \eqref{eq1} we hence obtain
\begin{align*}
\eqref{eq1} & = 
\sum_{\ell=1}^{k+1} \int_{\p^{n-1}} \frac{(-t)^{k+1-\ell}}{(k+2-\ell)!} \frac{e(\Omega_{\p} \otimes \Ft)}{t} \frac{\xi^{\ell-1}}{(\ell-1)!} \\
& = \frac{1}{t} \sum_{\ell=0}^{k} \frac{(-1)^{k+\ell}}{ (k+1-\ell)!} \int_{\p^{n-1}} t^{k+1-\ell} e(\Omega_{\p} \otimes \Ft) \frac{\xi^{\ell}}{\ell!} \\
& = t^{k-1} \sum_{\ell=0}^{k} \frac{(-1)^{n-1+k}}{ (k+1-\ell)!} \frac{1}{\ell!} \binom{n}{\ell+1}.
\end{align*}

Observe that for all $k \geq 0$ we have $(1+x)^k = \sum_{i \geq 0} \binom{k}{i} x^i$.
Using $\binom{-k}{i} = (-1)^i \binom{k+i-1}{i}$ we get in particular that for $r \geq 1$ we have
\[ \frac{x^{r-1}}{(1-x)^r} = \sum_{n \geq r-1} \binom{n}{r-1} x^n. \]
Therefore
\begin{align*}
\sum_{n} (-p)^n \int_{[ P_{n}(X,\beta) ]^{\red}} \tch_{2+k}(F)
& = t^{k-1} \sum_{\ell = 0}^{k} \frac{(-1)^{k-1}}{ (k+1-\ell)! \ell!} \left( \sum_{n \geq 0} \binom{n}{\ell + 1} p^n \right) \\
& = t^{k-1}  (-1)^{k-1} \sum_{\ell = 0}^{k} \frac{1}{ (k+1-\ell)! \ell!} \frac{p^{\ell+1}}{(1-p)^{\ell+2}} \\
& = - t^{k-1}  \frac{p}{(1-p)^2} \left( \frac{(-1)^k}{(k+1)!} \frac{ 1-p^{k+1}}{ (1-p)^{k+1} } \right).
\end{align*}
\end{proof}

\subsection{Descendents of fiber class}
Let $S \to \p^1$ be again an elliptic K3 surface with fiber class $F$.
The first evaluation we consider is the series
\begin{align*}
A_k(p,q) & := \frac{\left\langle \ch_{2+k}(F) \right\rangle^{\PT}}{ \langle 1 \rangle^{\PT}} \\
& = -\Theta^2 \Delta \left\langle \ch_{2+k}(F) \right\rangle^{\PT}.
\end{align*}
The holomorphic anomaly equation of Conjecture~\ref{conj:HAE K3xC} in this case reads as follows\footnote{Observe that we follow the convention of Remark~\ref{rmk:convention}.}:
\begin{lemma}
Assuming Conjecture~\ref{conj:HAE K3xC}, $A_k$ is a quasi-Jacobi form of index 0 and weight $k$ satisfying
\begin{gather*}
\frac{d}{dA} A_k = A_{k-1} \\
\frac{d}{dG_2} A_k = - \sum_{m_1 + m_2 = k-2} \frac{m_1! m_2!}{(k-1)!} A_{m_1} A_{m_2}
\end{gather*}
\end{lemma}
\begin{proof}
The class of the diagonal in $H_{\BG_m}(\BC \times \BC)$
with respect to the diagonal $\BG_m$-action is\footnote{Similarly,
the equivariant diagonal of $\p^1$ is
$[ \Delta_{\p^1} ] = \frac{1}{t} [(0,0)] + \frac{1}{-t} [(\infty, \infty)]$.}
\[ \Delta_{\BC} = \frac{1}{t} [ (0,0) ]. \]
We then find
\[ \pi_1^{\ast}([0] F) \cdot \Delta_{B} = \pi_1^{\ast}([0] \cdot [ \Delta_{\p^1} ] \cdot F_1 (F_1+F_2) = F_1 F_2 [(0,0)]. \]
Note also $\ch_k(F) = \widetilde{\ch}_k(F)$.
Hence we apply the $\frac{d}{dA}$-holomorphic anomaly equation, with output given by the above insertion. This yields the first claim.

For the second claim, one uses
\[ \CE(K_X F) = F K_X \CE(1) = K_X F_1 F_2 = -F_1 F_2 \quad (\text{with } t=1). \qedhere \]
\end{proof}

By Proposition~\ref{prop:leading term} we also have for $k \geq 0$ the leading term:
\[ A_k = \frac{(-1)^k}{(k+1)!} \frac{1-p^{k+1}}{(1-p)^{k+1}} + O(q). \]

\begin{example}
Assume Conjecture~\ref{conj:HAE K3xC}. Then the above constraints, together with the GW/PT correspondence and basic computations on the GW side along the lines of \cite{MPT, K3xP1} yield the following evaluations:
\begin{align*}
A_0 & = 1 \\
A_1 & = A \\
A_2 & = - G_2 + A^2/2 \\
A_3 & = - G_2 A + A^3/6 \\
A_4 & = \frac{1}{24} A^{4} - \frac{1}{2} A^{2} G_{2} + \frac{1}{3} G_{2}^{2} - \frac{1}{72} G_{4} \\
A_5 & = \frac{1}{120} A^{5} - \frac{1}{6} A^{3} G_{2} + \frac{1}{3} A G_{2}^{2} - \frac{1}{72} A G_{4}.
\end{align*}
\end{example}

\subsection{Descendents of point}
For $\pt \in H^4(S)$ the point class we consider the series
\[ B_k(p,q) := \frac{\left\langle \ch_{k}(\pt) \right\rangle^{\PT}}{\langle 1 \rangle^{\PT}}. \]

Assuming the holomorphic anomaly equation, then we see that $B_k$ is a quasi-Jacobi form of index $0$ and weight $k$ satisfying
\begin{gather*}
\frac{d}{dA} B_k = B_{k-1} + A_{k-1} \\
\frac{d}{dG_2} B_k = \frac{-2}{(k-1)!} \sum_{m_1 + m_2 = k} (m_1-2)! m_2! A_{m_1-2} B_{m_2}.
\end{gather*}
The constant term is easily fixed to be:
\[ B_k = 0 + O(q), \quad k>0. \]
Assuming Conjecture~\ref{conj:HAE K3xC} and using Gromov-Witten computations
one finds:
\begin{align*}
B_0 & = \left( \int_{S} -\pt \right) \Theta^2 \Delta(q) \left\langle 1 \right\rangle^{S \times \BC}_{\PT} = -1. \\
B_1 & = 0 \\
B_2 & = \frac{1}{2} A^2 - \frac{1}{2} \wp + 2 G_2 \\
B_3 & = \frac{1}{3} A^{3} + A G_{2} - \frac{1}{2} A \wp - \frac{1}{12} \wp' \\
B_4 & = \frac{1}{8} A^{4} - \frac{1}{4} A^{2} \wp -  G_{2}^{2} + \frac{1}{3} G_{2} \wp - \frac{1}{24} \wp^2 - \frac{1}{12} A \wp' + \frac{5}{36} G_{4}
\end{align*}

\subsection{$C$-series}
Consider classes $\alpha_1, \alpha_2 \in H^2(S,\BQ)$ such that
\begin{gather*}
\alpha_i \cdot F = \alpha_i \cdot B = 0 \text{ for } i=1,2  \\
\alpha_1 \cdot \alpha_2 = 1, \quad \alpha_1^2 = 0, \quad \alpha_2^2 = 0.
\end{gather*}
Consider the series
\[ C_{k\ell} := \frac{\left\langle \tch_{2+k}(\alpha_1) \tch_{2+\ell}(\alpha_2) \right\rangle^{\PT}_{1}}{\left\langle 1 \right\rangle^{\PT}} \]

\begin{lemma}
Assume Conjecture~\ref{conj:HAE K3xC}. Then
$C_{k \ell}$ is a quasi-Jacobi form of index 0 and weight
$k+\ell+2$ satisfying
\begin{gather*}
\frac{d}{dA} C_{k \ell} = C_{k-1,\ell} + C_{k, \ell-1} 
\end{gather*}
and
\begin{align*}
\frac{d}{dG_2} C_{k \ell} = & -2 \sum_{\substack{m_1 + m_2 = k}} \frac{ (m_1 - 1)! m_2!}{k!} A_{m_1-1} C_{m_2-1,\ell}  \\
& -2 \sum_{\substack{m_1 + m_2 = \ell}} \frac{ (m_1 - 1)! m_2!}{\ell!} A_{m_1-1} C_{k, m_2-1}   \\
& + 2 \binom{k+\ell}{k,\ \ell} A_{k + \ell}
-2 A_{k} A_{\ell}
\end{align*}
\end{lemma}

\begin{example}
Conjecture~\ref{conj:HAE K3xC} with the above constraints then yields:
\begin{align*}
C_{k\ell} & = 0 + O(q) \\
C_{k0} & = C_{0 \ell} = 0 \\
C_{11} & = D_{\tau}(G_2) \\
C_{21} & = D_{\tau}(G_2) A \\
C_{31} & = D_{\tau}(G_2) \frac{A^2}{2} + \frac{4}{3} G_2^3 - \frac{2}{3} G_2 G_4 + \frac{7}{720} G_6 \\
C_{41} & = D_{\tau}(G_2) \frac{A^3}{6} + \left( \frac{4}{3} G_2^3 - \frac{2}{3} G_2 G_4 + \frac{7}{720} G_6 \right) A \\
C_{22} & = D_{\tau}(G_2) A^2 + \frac{2}{3} G_2^3 - \frac{1}{6} G_2 G_4 - \frac{7}{720} G_6 \\
C_{32} & = \frac{1}{2} D_{\tau}(G_2) A^3 + ( 2 G_2^3 - \frac{5}{6} G_2 G_4)A 
\end{align*}
\end{example}

\subsection{A general formula for the stationary theory}
\label{subsec:general formula}
We now turn to what can be said about the general stationary theory of $S \times \BC$.
%
%
Define the partition function:
\[ Z_{S \times \BC}(\gamma_0, \gamma_1, \ldots ) := \left\langle \exp\left( \sum_{k \geq 0} \tch_{2+k}(\gamma_k) \right) \right\rangle^{\PT} \]
where $\gamma_0, \gamma_1, \gamma_2, \ldots \in H^{\ast}(S)$ are formal classes.
$Z_{S \times \BC}(\gamma_0, \gamma_1, \ldots )$ encodes all Pandharipande-Thomas invariants of $S \times \BC$.
%
We make the following basic dependence conjecture:
\begin{conj}
Assume that $\deg(\gamma_i) > 0$ for all $i$. Then we have
\[ Z_{S \times \BC}(\gamma_0, \gamma_1, \ldots ) = 
\exp\left( 
\sum_{k \geq 0} (\gamma_k, W + D_{\tau} F) A_k 
+ \sum_{k \geq 0} (\gamma_k, 1) B_k
+ \frac{1}{2} \sum_{i,j \geq 0} (\gamma_i \cdot \gamma_j) C_{ij} \right) \frac{-1}{\Theta^2 \Delta}.
\]
where $D_{\tau}$'s in the formula above stand for commuting the operators to the left and
applying them to the remaining terms. Moreover, $(\gamma_1, \gamma_2) = \int_{S} \gamma_1 \cup \gamma_2$ is the intersection pairing.
\end{conj}

A similar dependence has been conjectured for the Gromov-Witten theory of the K3 surface
in \cite{GWK3} and motivated the above conjecture.

\subsection{Conjectural evaluation of ABC} \label{subsec:ABC}
Finally, we conjecture a general expression for the series $A_k, B_k, C_k$.
For that recall the series
\[ \A(p,q) = D_p \log \Theta(p,q)
= -\frac{1}{2} -\frac{p}{1-p} - \sum_{n \geq 1} \sum_{d|n} (p^d - p^{-d}) q^n.
 \]
The series can naturally be expanded in $p=e^{z}$.
By some abuse of notation, we will write
\[
\A(z) := \A(p,q)|_{p=e^z}
\]
for this series.
Concretely, one has
\[
\A(z) = \frac{1}{z} - 2 \sum_{k \geq 1} G_{k}(q) \frac{z^{k-1}}{(k-1)!}.
\]
Below we will implicitly use the variable change $p=e^z$.

Our computations above lead to the following conjecture:

\begin{conj} For all $k \geq 0$, under the variable change $p=e^z$ we have:
\begin{align*}
A_k(p,q) & = \frac{1}{(k+1)!} \mathrm{Res}_{x=0} (A(z) + A(x))^{k+1} \\
B_k(p,q) & = \mathrm{Res}_{x=0} \left( \frac{( A(z) + A(x) )^k}{k!} ( A(x+z) - A(x) ) \right) \\
C_{k\ell}(p,q) & = \Res_{x_1=0} \Res_{x_2=0}\left( \frac{(A(x_1) + A(z))^{k+1}}{(k+1)!} \frac{(A(x_2) + A(z))^{\ell+1}}{(\ell+1)!} A'(x_1 - x_2) \right)
\end{align*}
where $\mathrm{Res}_{x=0}$ stands for taking the $x^{-1}$ coefficient,
and $A'(z) = \frac{d}{dz} A(z)$.
\end{conj}

\begin{example}
To compute $B_k$ and $C_{k \ell}$ one uses the basic expansion
\[ A(x + z) = e^{x \frac{d}{dz} } A(z) = A(z) + x A'(z) + \frac{x^2}{2} A''(z) + \ldots. \]
For example, if $\ell = 0$ we have
\[ \Res_{x_2=0} \left( A(x_2) + A(z) \right) ( A'(x_1) - x_2 A''(x_1) + \ldots ) = A'(x_1) \]
and hence correctly
\[ C_{k0} = \mathrm{Res}_{x_1=0}\left( \frac{(A(x_1) + A(z))^{k+1}}{(k+1)!} A'(x_1) \right) = \frac{1}{(k+2)!} \mathrm{Res}_{x_1=0} \left( (A(x_1) + A(z))^{k+2} \right)' = 0. \]
In case $\ell=1$ we get
\[ C_{k 1} = \mathrm{Res}_{x_1=0} \frac{(A(x_1) + A(z))^{k+1}}{(k+1)!} \left( -\frac{1}{x_1^3} - \frac{A(z)}{x_1^2} - 2 G_2 A(z) + G_4 x_1 + \ldots \right) \]
For example
\begin{multline*} C_{11} = \Res_{x_1=0}\Bigg[ \left( 1/2 x_1^{-2} + A x_1^{-1} + (A^2/2 - 2 G_2 ) - 2 A G_2 x + \ldots \right) \\ \times \left( \frac{-1}{x_1^3} - \frac{A(z)}{x_1^2} - 2 G_2 A(z) + G_4 x_1 + \ldots \right) \Bigg]
= -2 G_2^2 + \frac{5}{6} G_4 = D_{\tau}(G_2). \end{multline*}
\end{example}

\section{$\pi$-stable pairs invariants} \label{sec:pi stable pairs}
\subsection{Overview}
The goal of this section is to extend the quasi-Jacobi form 
conjecture (Conjecture~\ref{conj:QJac introduction}) to elliptic threefolds with non-vanishing $c_3(T_{X} \otimes \omega_X)$.
The idea is to use {\em $\pi$-stable pair invariants},
which is a version of Pandharipande-Thomas invariants
introduced in \cite{FM}
adapted to the elliptic fibration structure $\pi : X \to B$.

Throughout this section, we let $\pi : X \to B$ be an elliptically fibered threefold
with a section and a Weierstra\ss model. 

\subsection{$1$-dimensional fiber sheaves} \label{subsec:stability}
Let $\Coh_{\leq 1}(X)$ be the full subcategory of the category of coherent sheaves on $X$
consisting of sheaves whose support has dimension at most $1$.
Let $L$ be an ample line bundle on $B$ and let $H$ be an ample line bundle on $X$ such that $H - \pi^{\ast}(L)$ is also ample. 
Define the slope function
\[ \mu : \Coh_{\leq 1}(X) \to \mathsf{S}:= (-\infty, \infty] \times (-\infty, \infty], \]
by
\[ \mu(F) = \left( \frac{\chi(F)}{ \ch_2(F) \cdot f^{\ast} L } , \frac{\chi(F)}{ \ch_2(F) \cdot H } \right). \]
We order $\mathsf{S}$ lexicographically (i.e. $(a,b) < (a', b')$ iff $a < a'$ or $a=a$, $b < b'$).
The function $\mu$ defines a stability condition on $\Coh_{\leq 1}(X)$, compare \cite[Sec.3, Lem. 39]{BS}.

In particular, $\mu(F) = (\infty ,a)$ for $a \in \BR$ if $F$ is a $1$-dimensional sheaves supported on fibers of $\pi$,
and $\mu(F) = (\infty, \infty)$ if $F$ is zero-dimensional.

We let $\BT \subset \Coh_{\leq 1}(X)$ be the smallest extension-closed full subcategory
which contains all $\mu$-semistable sheaves $F$ of slope $\mu(F) > (\infty, 0)$.
We define the complement
\[ \BF= \{ F \in \Coh_{\leq 1}(X) | \Hom(T,F) = 0 \text{ for all } T \in \BT \}. \]
The pair $(\BT, \BF)$ defines a torsion pair on $\Coh_{\leq 1}(X)$.
The category $\BT$ is closed under taking quotients (use the stability condition),
and $\BF$ is closed under taking subobjects.

\begin{rmk} \label{rmk:T characterization by FM}
Let $p_1, p_2 : X \times_{B} X \to X$ be the two projections, let $\CP$ be the normalized relative Poincar\'e bundle on $X \times_B X$,
and consider the Fourier-Mukai transform
\[ \phi_{\CP} = p_{2 \ast}( p_1^{\ast}( - ) \otimes \CP) : D^b(X) \to D^b(X) \]
(all functors are derived). Then one has the characterization 
\[ \BT= 
\left\{ 
F \in \Coh_{\leq 1}(X) \middle| 
\begin{array}{c} F \text{ supported on fibers of } \pi, \\ \phi_{\CP}(F) \text{ is a sheaf } \end{array}
 \right\}, \]
see for example \cite{FM}.
\end{rmk}

\subsection{$\pi$-stable pairs}
We give the definition of $\pi$-stable pairs.\footnote{In fact,
our definition of $\pi$-stable pair differs slightly from \cite{FM}.
In \cite{FM} the $\pi$-stable pairs are defined with respect to the torsion pair $\BT'$
given by semistable sheaves of slope $> (\infty, -1)$.
The definition below is better behaved because every $\pi$-stable pair can be written as $\CO_X \to F$
and not just as a abstract $2$-term complex in the derived category.}

\begin{defn}[\cite{FM}]
A $\pi$-stable pair on $\pi : X \to B$ is a pair $(F,s)$ consisting of
\begin{itemize}
\item a coherent sheaf $F \in \BF$,
\item a section $s : \CO_X \to F$ with cokernel in $\BT$.
\end{itemize}
\end{defn}

In the following lemma we show that a $\pi$-stable pair is uniquely determined by the associated $2$-term complex $\CO_X \to F$ 
up to quasi-isomorphism.
Because of that, we will always
identify $\pi$-stable pairs with the corresponding two-term complexes.

\begin{lemma}
The category of $\pi$-stable pairs is equivalent to the category of
complexes $I^{\bullet}$ in the derived category $D^b(X)$ satisfying the following properties:
\begin{itemize}
\item $\ch(I^{\bullet}) = (1,0, -\beta, -m)$ for some $\beta \in H_2(X,\BZ)$ and $m \in \BQ$,
\item $h^i(I^{\bullet}) = 0$ for $i\neq 0,1$,
\item $h^0( I^{\bullet} )$ is torsion free and $h^1(I^{\bullet})$ lies in $\BT$,
\item $\Hom( Q[-1], I^{\bullet} ) = 0$ for every $Q \in \BT$.
\end{itemize}
\end{lemma}
\begin{proof}
Let $I^{\bullet}$ be a complex satisfying the above conditions.
We first prove that $\Hom(I^{\bullet}, \CO_X) = \BC$.
Since $h^0(I^{\bullet})$ is torsion-free of rank $1$ and has vanishing first Chern class, it is isomorphic to the ideal sheaf
$I_C$ of a curve $C$.
Consider the exact triangle
\[ I_C \to I^{\bullet} \to Q[-1] \to I_C[1], \]
for some $Q \in \BT$, and apply $\Hom( - , \CO_X)$. We see that
\[ 0 \to \Hom(I^{\bullet}, \CO_X) \to \Hom(I_C, \CO_X) \to \Hom(Q[-2], \CO_X). \]
We then have
\[ \Hom( Q[-2], \CO_X ) = \Ext^1( \CO_X, Q \otimes \omega_X ) = H^1(X,Q). \]
By using the Harder-Narasimhan flitration we may assume that $Q$ is stable and of the form $Q = i_{s \ast}Q'$ for some stable sheaf $Q' \in \Coh(X_s)$ of slope $>0$ for some $s \in B$.
But then it follows that $H^1(X,Q) = 0$ because $\phi_{\CP}(Q)$ is a sheaf by Remark~\ref{rmk:T characterization by FM}.
Hence $\Hom(I^{\bullet}, \CO_X) = \Hom(I_C, \CO_X) = \BC$. Let $F = \mathrm{Cone}( I^{\bullet} \to \CO_X )$
be the cone of the canonical morphism.
This fits into the exact sequence
\[ I^{\bullet} \to \CO_X \to F. \]
Since $h^0(I^{\bullet}) = I_C \to \CO_X$ is injective, we have that $F$ is a $1$-dimensional sheaf.
The cokernel of $s:\CO_X \to F$ lies in $\BT$ since $h^1(I^{\bullet}) \in \BT$.
Moreover, for any $T \in \BT$ applying $\Hom(T, -)$ shows that $\Hom(T,F) = \Hom( T, I^{\bullet}[1] ) = 0$ so $(F,s)$ is a $\pi$-stable pair.

Conversely, for any $\pi$-stable pair $(F,s)$ the $2$-term complex $I^{\bullet} = [\CO_X \to F]$ satisfies the above conditions.
\end{proof}

\subsection{Moduli space} \label{subsec:moduli space}
Let $P^{\pi}_{n}(X,\beta)$ be the moduli functor of $\pi$-stable pairs $I = [\CO_X \to F]$ satisfying
\[ \ch_2(F) = \beta \in H_2(X,\BZ), \quad \chi(F) = n \in \BZ. \]
The objects of $P^{\pi}_{n}(X,\beta)$ over a scheme $S$ are the two-term complexes $[\CO_{X \times S} \to \CF]$ with $\CF$ flat over $S$ such that
for all geometric points $s \in S$ the restriction $[\CO_{X,s} \to \CF_{s}]$ is a $\pi$-stable pair with the given numerical data.

\begin{thm} \label{Ppi representable} $P^{\pi}_{n}(X,\beta)$ is represented by a proper algebraic space.
\end{thm}

For the proof we apply the Artin representability theorem.
We follow the same strategy as for Bryan-Steinberg pairs \cite{BS}
for which the representability of the moduli space was proven in \cite[Sec.3.2]{Tudor}.
In particular, we need to prove
that the moduli functor $P^{\pi}_{n}(X,\beta)$ is open, bounded, separated, complete, and has trivial automorphism.
The openness follows immediately from the openness of semistable sheaves
(compare \cite[Prop.3.4]{Tudor}). Moreover,
as observed in \cite[Lemma 4]{FM} the boundedness is proven as in \cite[Sec. 4.2]{Toda}.
The next step are the trivial automorphisms and negative Ext's:

\begin{lemma} For any $\pi$-stable pair $I^{\bullet}$, we have $\Ext^{<0}(I^{\bullet}, I^{\bullet}) = 0$.
\end{lemma}
\begin{proof}
Apply $\Hom( I^{\bullet}, - )$ to $I_C \to I^{\bullet} \to Q[-1]$.
Then show $\Ext^{<0}(I^{\bullet}, I_C) = 0$ by applying $\Hom( -, I_C)$ to the same sequence. Similar for $Q[-1]$.
\end{proof}

\begin{lemma} For any two $\pi$-stable pairs $I^{\bullet}, J^{\bullet}$,
the canonical map $\Hom(I^{\bullet}, J^{\bullet}) \to \Hom(I^{\bullet}, \CO_X) = \Hom(\CO_X, \CO_X)$ is injective.
In particular, $\Hom(I^{\bullet}, I^{\bullet}) = \BC \id$.
\end{lemma}
\begin{proof}
Same proof as in \cite[Prop.3.2]{Tudor}.
\end{proof}

By arguing as in \cite[Prop.3.6]{Tudor}
the two lemmata above imply the separatedness of the moduli functor $P^{\pi}_{n}(X,\beta)$.
We hence have to tackle the properness:

Let $R$ be a discrete valuation ring with fraction field $K$, residue field $k$, and uniformizer $\pi$.
Let $X_R = X \times \Spec(R)$ and $X_{K} = X \times \Spec(K)$.

\begin{prop} Given a $\pi$-stable pair $I = [\CO_{X_K} \to F]$ over $K$, there exists
a $\pi$-stable pair $\CI = [\CO_{X_R} \to \CF]$ over $R$ such that $\CI_{K} = I$.
\end{prop}

\begin{proof}
The argument is a modification of Langton's original semistable reduction, compare \cite[Sec.2.B]{HL}.
For the modification we follow \cite[Prop.3.7]{Tudor} with some exceptions.
(We also refer to \cite{Lo} for similar arguments in a related case.)

Let $\CH$ be a $R$-flat extension of $F$ (flat over $R$ means that $\CH$ is locally torsion-free over $R$,
so just extend $F$ to any coherent sheaf and quotient out the $R$-torsion).

\vspace{4pt}
\textbf{Step 1.} We want to find a subsheaf $\CH' \subset \CH$ such that the restriction $\CH'_k$ lies in $\BF$.
Assume that $\CH_k$ does not lie in $\BF$.
Then there exists a maximally destabilizing subsheaf $Q_0 \subset \CH_k$ with $Q_0 \in \BT$.
In particular, $Q_0$ is semistable and all Harder-Narasimhan factors of $\CH_k$ with respect to $\mu$-stability have smaller slope then $Q_0$.
Consider the sequence
\[ 0 \to Q_0 \to \CH_k \to P_0 \to 0. \]
Let $\iota : X_k \to X_R$ be the inclusion and define the kernel
\[ \CH^1 = \Ker( \CH \to \iota_{\ast} \CH_k \to \iota_{\ast} P_0 \to 0 ). \]
We obtain the short exact sequence
\[ 0 \to \CH^1 \to \CH \to P_0 \to 0. \]
Restricting to $X_k$ and noting $\mathrm{Tor}_k(P_0, k) = P_0$ we obtain
\[ 0 \to P_0 \to \CH^1_k \to \CH_k \to P_0 \to 0 \]
and hence the short exact sequence
\[ 0 \to P_0 \to \CH^1_k \to Q_0 \to 0. \]
We see that $\CH^1_k$ is an extension of $P_0, Q_0$ in the opposite way, that is we have flipped the extension.
In particular,
$\CH_k$ and $\CH^1_k$ have the same Chern character, same Hilbert polynomial, and same $\mu$-slope.

We want to argue that $\CH^1_k$ brings us closer to a sheaf in $\BF$.
If $\CH^1_k$ lies in $\BF$ we are done. Otherwise, consider again a maximally destabilizing subsheaf
\[ 0 \to Q_1 \to \CH^1_k \to P_1 \to 0. \]
We form the diagram of rows of short exact sequences:
\begin{equation} \label{comm diag}
\begin{tikzcd}
0 \ar{r} & L \ar{r} \ar{d} & Q_1  \ar{d} \ar{r} \ar{dr}{f} & \mathrm{Im}(f) \ar{d}\ar{r} & 0 \\
0 \ar{r} & P_0 \ar{r} & \CH^1_k \ar{r} & Q_0 \ar{r} & 0 
\end{tikzcd}
\end{equation}
where $f : Q_1 \to \CH^1_k \to Q_0$ is the composition and $L = \mathrm{Ker}(f)$.

If $\mu(Q_1) > \mu(Q_0)$ then $f$ must vanish by stability, so $L=Q_1$;
moreover, the injection $L \hookrightarrow P_0$ together with the fact that the largest slope of a Harder-Narasimhan factor of $P_0$ is strictly smaller than $\mu(Q_0) < \mu(Q_1)$,
then shows that $L=0$, which is a contradiction with the choice of $Q_1$.
Hence $\mu(Q_1) \leq \mu(Q_0)$.
We now iterate the above process. If it does not stop, we obtain a sequence
\[ \mu(Q_0) \geq \mu(Q_1) \geq \mu(Q_2) \geq \ldots .\]
Since each $Q_n$ is a destabilizing subobject of $\CH^n_k$ we must have $\mu(Q_n) > \mu(\CH^n_k) = \mu(\CH_k)$.
Moreover, the support of $Q_n$ lies in the support of $\CH^n_k$; since $\ch_2(\CH^n_k)$ is independent of $n$, we find that $\ch_2(Q_n)$ lies in a bounded set,
so the denominators of $\mu(Q_n)$ lie in a finite set. We find that there exists an $n_0 \geq n$ such that
$\mu(Q_n) = \mu(Q_{n_0})$ for all $n \geq n_0$.
We assume that $n_0 = 0$, and hence $\mu(Q_n) = \mu(Q_0)$ for all $n$.

We return to the diagram \eqref{comm diag}.
Since $\mu(Q_0) = \mu(Q_1)$ we see that $L$ is semi-stable of slope $\mu(Q_0)$ and hence lies in $\BT$ (use that the category of semi-stable sheaves of a given slope is abelian);
the injection $L \to P_0$ then implies $L=0$ by the torsion-pair property, so that $Q_0 \subset Q_1$, and hence
\[ Q_0 \subset Q_1 \subset Q_2 \subset \ldots . \]
Since $\ch_2(Q_n)$ is bounded, it must stabilize at some point $n_0$, at which point we get $Q_{n} = Q_{n+1} = \ldots $ since the slope has also stabilized.
By shifting indices (assuming $n_0 = 1$) we hence can assume that $Q_0 = Q_1 = \ldots $.
This implies that the second row in \eqref{comm diag} splits and hence
\[ \CH^n_k = Q_n \oplus P_n \cong Q_{n-1} \oplus P_{n-1} \]
and the summands $Q_n$ and $Q_{n-1}$ coincide. Hence also $P_{0} = P_1 = \ldots $.

We now argue for a contradiction. The defining sequence
\[ 0 \to \CH^n \to \CH^{n-1} \to \iota_{\ast} P_{n-1} \to 0 \]
and $P_0 = P_1 = \ldots $ shows that for every $n$ we have a sequence
\[ 0 \to \iota_{\ast} P_{0} \to \CH / \CH^n \to \CH / \CH^{n-1} \to 0 \]
and hence $(\CH/\CH^n)_k = (\CH^0/\CH^1)_k = P_0$.
This shows that $\CH/\CH^n$ is flat over $R/\pi^n$, see \cite[2.1.3]{HL}, and the natural morphism
$\CH/\pi^n \CH \to \CH/\CH^n$ (obtained from the inclusion $\pi^n \CH \subset \CH^n$) is hence an $R/\pi^n$-flat quotient.
Let $p$ be the Hilbert polynomial of $P_0$. It follows that the base change of the proper map
$\mathrm{Quot}_{X_R/R}( \CH, p ) \to \Spec(R)$ to $\Spec(R/\p^n)$ is surjective for all $n$.
Hence $\mathrm{Quot}_{X_R/R} \to \Spec(R)$ is surjective.
In particular, there exists a surjection $\CH \to \CP_0$ with $\CP_0$ flat over $R$.
The kernel $\CQ_0 = \mathrm{Ker}(\CH\to \CP_0)$ is a flat family of sheaves in $\BT$, and
$\CQ_K \subset \CH_K$ yields a non-trivial subobject in $\CH_K \in \BF_K$, a contradiction.

\vspace{4pt}
\textbf{Step 2. (Extending the section)}
By Step 1 there exists a subsheaf $\CH' \subset \CH$ (automatically flat over $R$) such that $\CH'_k$ lies in $\BF$.
We replace $\CH$ by such a $\CH'$ and hence from now on assume that $\CH_k$ lies in $\BF$.

The $R$-module $\Hom_{X_R}(\CO_{X_R}, \CH)$ has a canonical section after tensoring with $K$.
After multiplying with a sufficiently high power of $\pi$ it follows that there exists a global section
\[ s : \CO_{X_R} \to \CH \]
estending the given section $\CO_{X_R} \to F$.
Let
\[ \CK= \mathrm{Coker}(s). \]
We want to argue now that $\CH$ and $s$ can be replaced so that $\CK_k$ lies in $\BT$.
We do this in two more steps.

\vspace{4pt}
\textbf{Step 3.} Let $\CK_f \subset \CK$ be the largest flat subsheaf of $\CK$ over $R$ (this exists since $R$ is a DVR),
and let $M = \CK/\CK_f$. We have that $M$ is supported over (a finite thickening of\footnote{We view $M$ as a $\CO_X[\epsilon]/\epsilon^m$ module for some $m>0$ below.}) the central fiber. Consider the unique exact sequence
\[ 0 \to A \to M \to B \to 0 \]
with $A \in \BT$ and $B \in \BF$ provided by the torsion pair. 
Define 
\[ \CH' = \mathrm{Ker}( \CH \to M \to B_k ). \]

Restricting $0 \to \CH' \to \CH \to \iota_{\ast} B_k \to 0$ to $k$ yield
\[ 0 \to \iota_{\ast} B_k \to \CH'_k \to \mathrm{Ker}(\CH_k \to \iota_{\ast} B_k) \to 0. \]
Since $\mathrm{Ker}(\CH_k \to \iota_{\ast} B_k)$ is a subobject of $\CH_k \in \BF$, it lies in $\BF$,
and hence $\CH'_k$ is an extensions of objects in $\BF$ and hence also lies in $\BF$.

Replacing $\CH$ by $\CH'$ and iterating this procedure above a finite number of times (the number is equal to the minimum $m$ such that $\pi^m B = 0$),
we can assume that the torsion part of $\CK$, i.e. $\CK / \CK_f$ lies in $\BT$.
%
%

\vspace{4pt}
\textbf{Step 4.} If $\CK_k$ is not in $\BT$, there exists a minimally destabilzing quotient $P_0$ of $\CK_k$
with $P_0 \in \BF$.
Consider the exact sequence
\[ 0 \to Q_0 \to \CK_k \to P_0 \to 0. \]
Define
\[ \CH^1 = \mathrm{Ker}( \CH \to \CH_k \to \CK_k \to P_0 \to 0 ). \]
The section $\CO \to \CH$ factors through $\CH^1$ with cokernel $\CK^1$.

Restricting $0 \to \CH^1 \to \CH \to \iota_{\ast} P_0 \to 0$ to $X_k$ yields
\[ 0 \to P_0 \to \CH^1_k \to \CH_k \to P_0 \to \]
Since $\CH_k \in \BF$ and $\BF$ is closed under subobjects, we have $\mathrm{Ker}(\CH_k \to P_0)$ lies in $\BF$,
and hence $\BH^1_k$ lies in $\BF$ since $\BF$ is closed under extension.

The cokernel $\CK^1$ fits into the short exact sequence
\[ 0 \to \CK^1 \to \CK \to \iota_{\ast} P_0 \to 0. \]
Restricting to $X_k$ and observing that $\CK$ does not have to be flat over $R$ yields the exact sequence (not necessarily exact on the left)
\[ P_0 \to \CK^1_k \to \CK_k \to P_0. \]
Hence we obtain the sequence
\[ P_0 \to \CK^1_k \to Q_0 \to 0. \]

If $\CK^1_k \in \BT$ we are done, otherwise pick again a minimally destabilizing quotient $P_1$ of $\CK^1_k$.
We then have again $P_1 \in \BF$. Consider the sequence
\[ 0 \to Q_1 \to \CK^1_k \to P_1 \to 0. \]
Consider the diagram
\begin{equation} \label{comm diag2}
\begin{tikzcd}
  & P_0 \ar{dr}{f} \ar{r} \ar{d} & \CK^1_k  \ar{d} \ar{r} & Q_0  \ar{d}\ar{r} & 0 \\
0 \ar{r} & \mathrm{Im}(f) \ar{r} & P_1 \ar{r} & L \ar{r} & 0.
\end{tikzcd}
\end{equation}
where $f : P_0 \to \CK^1_k \to P_1$ and $L = \mathrm{Coker}(f)$.

If $\mu(P_1) < \mu(P_0)$ then $f=0$ by stability and hence $L=P_1$.
Since every Harder-Narasimhan factor of $Q_0$ has slope strictly bigger than $\mu(P_0)$, and we have the quotient $Q_0 \to L$, and we get $L=0$, which is a contradiction to the choice of $P_1$.
Hence $\mu(P_1) \geq \mu(P_0)$.
We now iterate again the above process, which gives a series of inequalities
\[ \mu(P_0) \leq \mu(P_1) \leq \mu(P_2) \leq \ldots \ . \]
Since $P_n$ is a quotient of $\CH_k^n$ and $\ch_2(\CH_k^n)$ is independent of $n$, we see that $\ch_2(P_n)$ lies in a finite set.
Moreover, $\mu(P_n) \geq \mu(\CH_k^n) = \mu(\CH_k)$. 
Hence there exists an $n_0$ such that $\mu(P_n) = \mu(P_{n+1})$ for all $n \geq n_0$. We assume $n_0=0$.
Then one finds that $L$ is semi-stable of the same slope as $\mu(P_0)= \mu(P_1)$,
and hence $L=0$ as a quotient of $Q_0$. This gives the series of surjections
$P_0 \twoheadrightarrow P_1 \twoheadrightarrow P_2 \twoheadrightarrow \ldots$ which has to stabilize at some point $n_0$.
We take $n_0 = 0$, so $P_0 = P_1 = \ldots$.

We argue now by contradiction. The sequence
\[ 0 \to \CK^n \to \CK^{n-1} \to \iota_{\ast} P_0 \to 0 \]
(by mapping to $0 \to \CK \to \CK \to 0 \to 0$ and applying the snake lemma) 
shows that we have the exact sequence
\[ 0 \to \iota_{\ast} P_0 \to \CK / \CK^n \to \CK / \CK^{n-1} \to 0, \]
and $\CK / \CK^1 = P_0$. We obtain that
$\CK / \CK^n$ is flat over $R / \pi^n$. Moreover, the quotient $\CK / \pi^n \CK \to \CK / \CK_n$ 
factors through $\CK_f / \pi^n \CK_f \to \CK/ \CK_n$ and then arguing as in Step 1 lifts to a quotient
$\CK_f \twoheadrightarrow \CP_0$ where $\CP_0$ is flat family over $R$ of semistable sheaves of slope $\mu(P_0)$.
Base changing to $K$ yields a non-trivial surjection $\CK_{K} \twoheadrightarrow (\CP_0)_K$ from an element of $\BT_K$ to $\BF_K$, a contradiction.
\end{proof}

\subsection{Invariants}
By work of Huybrechts and Thomas \cite{HT} the moduli space
of $\pi$-stable pairs carries a virtual fundamental class
\[
\left[ P^{\pi}_{n}(X,\beta) \right]^{\vir} \in A_{\mathrm{vd}}(P^{\pi}_{n}(X,\beta)), \quad \mathrm{vd} = \int_{\beta} c_1(X).
\]
Let $\CO \to \BF$ denote the universal $\pi$-stable pair on $P_n^{\pi}(X,\beta) \times X$.
Descendent classes on the moduli space are defined as before for $\gamma \in H^{\ast}(X)$ and $k \in \BZ$ by
\[ \ch_k(\gamma) = \pi_{P \ast}( \pi_X^{\ast}(\gamma) \cdot \ch_k( \BF - \CO )) \in H^{\ast}(P_n^{\pi}(X,\beta)). \]
Define the $\pi$-stable pair invariants by
\[
\left\langle \ch_{k_1}(\gamma_1) \cdots \ch_{k_n}(\gamma_n) \right\rangle^{X, \piPT}_{n,\beta}
=
 \int_{[ P^{\pi}_{n} (X,\beta) ]^{\text{vir}}} \prod_i \ch_{k_i}(\gamma_i).
\]

For any fixed class $\beta \in H_2(B,\BZ)$ let
\[
\left\langle \ch_{k_1}(\gamma_1) \cdots \ch_{k_n}(\gamma_n) \right\rangle^{X, \piPT, \pi}_{\beta}
= 
\sum_{\substack{\widetilde{\beta} \in H_2(X,\BZ) \\ \pi_{\ast} \widetilde{\beta} = \beta }}
\sum_{m \in \frac{1}{2} \BZ}
i^{2m} p^{m}  q^{W \cdot \widetilde{\beta}} \left\langle \ch_{k_1}(\gamma_1) \cdots \ch_{k_n}(\gamma_n) \right\rangle^{X,\piPT}_{m+\frac{1}{2} d_{\widetilde{\beta}}, \widetilde{\beta}}
\]
where $i = \sqrt{-1}$ and $d_{\beta} = \int_{\widetilde{\beta}} c_1(T_X)$.
We also define the normalized series
\[
Z_{\beta}^{\piPT}\left( \ch_{k_1}(\gamma_1) \cdots \ch_{k_n}(\gamma_n) \right)
:= 
\frac{ \left\langle \ch_{k_1}(\gamma_1) \cdots \ch_{k_n}(\gamma_n) \right\rangle^{X, \piPT, \pi}_{\beta} }{\left\langle 1 \right\rangle^{X, \piPT, \pi}_{0} }
\]

We can compute the normalization factor here explicitly:
\begin{prop} \label{prop:pi PT normalization}
\[ \left\langle 1 \right\rangle^{\piPT,\pi}_{\beta=0} = \prod_{m \geq 1} (1-q^m)^{-e(B) - c_1(N) \cdot (c_1(T_B) + c_1(N))} \]
where $N = N_{B/X}$ is the normal bundle of the section.
\end{prop}
\begin{proof}
Let $[\CO_X \to F] \in P^{\pi}_{n}(X, d \mathsf{f})$ where $\mathsf{f} \in H_2(X,\BZ)$ is the fiber class of the fibration $X \to B$.
Since $F \in \BF$ all Harder-Narasimhan factors of $F$ must have holomorphic Euler characteristic $\chi \leq 0$.
This shows that $n \leq 0$.
On the other hand, $\CO_X \to F$ factors as $\CO_X \twoheadrightarrow \CO_C \hookrightarrow F$ for a curve of the form $C=\pi^{-1}(z)$ for a $0$-dimensional subscheme $z \subset B$.
The cokernel $Q = \mathrm{Coker}(\CO_C \to F)$ lies in $\BT$ hence we get that also $\chi(F) = \chi(\CO_C) + \chi(Q) \geq 0$.
Hence $F = \CO_C$ and $n=0$, and $\CO_X \to F$ is an ordinary stable pair.
We hence see that
\begin{equation} \label{dadf}
P^{\pi}_{n}(X, d \mathsf{f}) = 
\begin{cases} 
\varnothing & \text{ if } n \neq 0 \\
P_{0}(X,d \mathsf{f} ) & \text{ if } n=0.
\end{cases}
\end{equation}
The claim hence follows from Lemma~\ref{lemma:basic computation}.
\end{proof}

Recall from Section~\ref{ex:C2xE} the equivariant threefold $\BC^2 \times E$.
We compute the corresponding series of $\pi$-PT invariants in a special case.
This example shows that Conjecture~\ref{conj:QJac pi-PT intro} can not be extended to the equivariant case.

\begin{prop} For arbitrary $t_1,t_2$ and $\gamma \in H^{\ast}(E)$ we have
\begin{align*}
    \sum_{k \geq 0} Z^{E\times \BC^2,\piPT}(\tch_k(\gamma))x^k &= -\left(\int_{E} \gamma\right) \frac{1}{t_1 t_2 S(i\sqrt{t_1 t_2}x)} \prod_{n\geq 1} \frac{(1-q^n)(1-q^n e^{-(t_1+t_2)x})}{(1-q^n e^{-t_1 x})(1-q^n e^{-t_2 x})}\\
    &=-\left(\int_{E}\gamma\right)\frac{1}{t_1t_2}\exp\left( -\sum_{i,j\geq1}(-1)^{i+j}\frac{t_1^it_2^j}{i!j!} z^{i+j}G_{i+j}^{\delta_{i,j}} \right),
\end{align*}
where 
\[
    G_k^s = -s\frac{ \text{B}_{k}}{k}+\sum_{d\geq 1} \frac{q^d}{1-q^d}d^{k-1}.
\]
Almost all non-zero $x$-coefficients on the left hand side are not quasi-Jacobi forms.
\end{prop}
\begin{proof}
By \eqref{dadf} 
the left hand side is given precisely by the expression in \eqref{4ygfg} in case $n=1$.
In this case, we can reuse the proof strategy of \cite[Thm.~6.5]{character}. This yields
\begin{align*}
\sum_{k \geq 0} Z^{E\times \BC^2,\pi}(\tch_k(\gamma))x^k 
    &=\prod_{n\geq 1} (1-q^n)\sum_{\lambda}q^{|\lambda|} \sum_{i\geq 1} e^{(-\lambda_it_2-it_1)x} \\
    & = \left\langle \sum_i e^{(-\lambda_it_2-it_1)x} \right\rangle_q \\
    &= \sum_{k\geq 1} e^{-t_1x}\frac{(q)_\infty}{(qe^{-t_2x})_\infty}\sum_{r\geq 0} e^{t_1rx}\frac{(qe^{-t_2x})_r}{(q)_r} \\
& = e^{-t_1x}\frac{(q)_\infty (qe^{(-t_1-t_2)x})_\infty}{(qe^{-t_2x})_\infty (e^{-t_1x})_\infty}
\end{align*}
where $(a)_n = \prod_{k=0}^n (1-aq^k)$ and 
\[
    \left<f\right>_q = (q)_\infty \sum_{\lambda} f(\lambda) q^{|\lambda|}.
\]
In the above computation we used \cite[Lem.~6.6]{character} and Heine's $q$-analog of the Gauss $_2 F_1$-summation, where for any $a,b,c$ with $|c|< |ab|$:
\[
    \sum_{n=0}^\infty \frac{(a)_n(b)_n}{(c)_n(q)_n}\left(\frac{c}{ab}\right)^n = \frac{(c/a)_\infty(c/b)_\infty}{(c)_\infty(c/ab)_\infty}
\]
This was used in the last equality with $a=\epsilon, b=qe^{-t_2x},c=qe^{(-t_1-t_2)x}\epsilon$ after taking the limit $\epsilon\to 0$.
The claim now follows.
\end{proof}

\end{document}